\documentclass[10pt]{article}
\usepackage{amsfonts}
\usepackage{amssymb}
\usepackage{amsthm}
\usepackage{amsmath}    
\usepackage{bbm}
\usepackage{graphicx}
\usepackage{latexsym}
\usepackage{mathrsfs}
\usepackage[all]{xy}

\usepackage{geometry}
\newcommand{\wa}{\hat}

\newcommand{\cA}{\mathcal{A}}
\newcommand{\cB}{\mathcal{B}}

\newcommand{\cD}{\mathcal{D}}

\newcommand{\cF}{\mathcal{F}}

\newcommand{\cH}{\mathcal{H}}

\newcommand{\cK}{\mathcal{K}}
\newcommand{\cL}{\mathcal{L}}

\newcommand{\cO}{\mathcal{O}}

\newcommand{\cT}{\mathcal{T}}

\newcommand{\cBH}{\mathcal{BH}}

\newcommand{\rC}{\mathrm{C}}

\newcommand{\sA}{\mathscr{A}}
\newcommand{\sB}{\mathscr{B}}

\newcommand{\sE}{\mathscr{E}}
\newcommand{\sF}{\mathscr{F}}

\newcommand{\sH}{\mathscr{H}}

\newcommand{\sK}{\mathscr{K}}

\newcommand{\sBHm}{\sB \tilde \sH}
\newcommand{\sKH}{\sK \tilde \sH}
\newcommand{\sBH}{\mathscr{BH}}

\newcommand{\tr}{\triangleright}

\newcommand{\Si}{\Sigma} 

\newcommand{\bC}{\mathbb{C}}

\newcommand{\bN}{\mathbb{N}}
\newcommand{\bR}{\mathbb{R}} 

\newcommand{\bZ}{\mathbb{Z}}

\newcommand{\si}{\sigma}

\newcommand{\eps}{\varepsilon}
\newcommand{\e}{\epsilon}

\DeclareMathAlphabet\EuFrak{U}{euf}{m}{n}	
\SetMathAlphabet\EuFrak{bold}{U}{euf}{b}{n}	

\newcommand {\ovl}{\overline}

\newcommand{\ad}{\mathrm{ad}}

\author{\textsc{Giuseppe Ruzzi$^{1}$ and Ezio Vasselli$^{2}$}
\footnote{Both the authors are supported by the  EU network ``Noncommutative Geometry" MRTN-CT-2006-0031962.}\\
  \null\\
\small{$^{1}$Dipartimento di Matematica, Universit\`a di Roma ``Tor Vergata'',}\\
\small{Via della Ricerca Scientifica, I-00133 Roma,  Italy.  \texttt{ruzzi@mat.uniroma2.it}} \\
\small{$^{2}$Dipartimento di Matematica, Universit\`a di Roma ``La Sapienza'',}\\
\small{Piazzale Aldo Moro 5, I-00185 Roma, Italy.  \texttt{ ezio.vasselli@gmail.com  }}\\
}
        
\title{\textsc{The $C_0(X)$-algebra of a net and index theory}}

\begin{document}
\maketitle

\begin{abstract}
Given a connected and locally compact Hausdorff space $X$ with a good base $K$
we assign, in a functorial way, 
a $C_0(X)$-algebra to any precosheaf of $\rC^*$-algebras $\cA$ defined over $K$.
Afterwards we consider the representation theory and the Kasparov $K$-homology of $\cA$, 
and interpret them in terms, respectively, of the representation theory 
and the $K$-homology of the associated $C_0(X)$-algebra.
When $\cA$ is an observable net over the spacetime $X$ in the sense of algebraic quantum field theory,
this yields a geometric description of the recently discovered representations affected by the topology of $X$.
\end{abstract}

\tableofcontents
\markboth{Contents}{Contents}
\newpage


  \theoremstyle{plain}
  \newtheorem{definition}{Definition}[section]
  \newtheorem{theorem}[definition]{Theorem}
  \newtheorem{proposition}[definition]{Proposition}
  \newtheorem{corollary}[definition]{Corollary}
  \newtheorem{lemma}[definition]{Lemma}

  \theoremstyle{definition}
  \newtheorem{remark}[definition]{Remark}
    \newtheorem{example}[definition]{Example}

\theoremstyle{definition}
  \newtheorem{ass}{\underline{\textit{Assumption}}}[section]


\numberwithin{equation}{section}

\section{Introduction}

Let $X$ be a space and $K$ denote the partially ordered set ({\em poset})
given by a base of $X$ ordered under inclusion. In the present work we proceed on
studying the (noncommutative) geometric invariants of $X$ encoded by $K$, 
following the research line of \cite{RRV09}.

\

The interest on this question arises from the algebraic formulation 
of quantum field theory over the (possibly curved) spacetime $X$, 
an approach where the basic mathematical objects are a net of $\rC^*$-algebras over $K$
and the set of its Hilbert space representations of physical interest, 
called the {\em sectors} of the net.
Geometric effects on quantum systems, the most famous of which is the Aharonov-Bohm effect, 
are known since the late fifties,
nevertheless topological invariants have been discovered in the analysis of sectors 
only in recent times (see \cite{BR09,CHKL10,Ruz05,BFM09}),
and since these are expressed in terms of the poset rather than the space, 
the question of representing them as genuinely geometric objects arises.

\

Nets of $\rC^*$-algebras appeared in quantum field theory in terms
of pairs $(\cA,\jmath)_K$, where $\cA := \{ \cA_Y \}$ is a family of $\rC^*$-algebras
indexed by the elements of $K$ and 
$\{ \jmath_{Y'Y} : \cA_Y \to \cA_{Y'}$, $Y \subseteq Y' \}$
is a family of $*$-monomorphisms fulfilling the relations 
\[
\jmath_{Y''Y'} \circ \jmath_{Y'Y} = \jmath_{Y''Y} \  , \qquad   Y \subseteq Y' \subseteq Y'' \ ,
\]
the basic idea being that each $\cA_Y$ is the $\rC^*$-algebra of quantum observables localized
in the region $Y \subset X$.
This notion can clearly be given in other categories, as the ones of (topological)
groups and Hilbert spaces; in recent times, a particular attention has been
given to {\em net bundles}, i.e. nets such that every $\jmath_{Y'Y}$ is an
isomorphism (\cite{RRV09}).

\

The reader experienced in algebraic topology can recognize that what we call a 
{\em net} is, as a matter of fact, a precosheaf
{\footnote{The term {\em net}, standard in algebraic quantum field theory, shall be used 
           in this paper as a synonym of {\em precosheaf}.}}, 
thus it is expectable that these objects actually 
encode non-trivial geometric invariants of $X$.
Some relevant results in this direction have been proved, showing a strong 
interplay at the level of homotopy. We just mention the following two facts:
(i)    The fundamental group $\pi_1(K)$ is isomorphic to the fundamental group $\pi_1(X)$ (\cite{Ruz05});
(ii)   The category of net bundles over $K$ is equivalent to the ones of locally constant bundles
       over $X$ and of $\pi_1(X)$-dynamical systems (\cite{RRV09,RRV12}). 
       In this last case the geometric meaning of the $\pi_1(X)$-action is given by holonomy 
       and is central in the theory of net bundles. 

\

In accord to the previous considerations we have a (noncanonical) equivalence between the category 
of $\rC^*$-net bundles and the one of locally constant $\rC^*$-bundles.
In the present work we show that any net of $\rC^*$-algebras
can be interpreted as a precosheaf of local sections of a $C_0(X)$-algebra.
This extends, this time in a canonical way, the equivalence proved at the level of
net bundles to a functor from the category of nets of $\rC^*$-algebras to the one
of $C_0(X)$-algebras. Consequently, the sectors studied in \cite{BR09,BFM09}
can be understood as $C_0(X)$-representations over locally constant bundles of Hilbert spaces. 
This will allow us to define, using the Cheeger-Chern-Simon character (\cite{CS85}), 
an additive characteristic class on the set of sectors generalizing the statistical dimension, 
a well-known invariant in algebraic quantum field theory (\cite{RVCS}).

\

As an application of the previous results we consider \emph{nets of Fredholm modules}, 
a generalization of a notion introduced in \cite[\S 6]{Lon01} to describe sectors in terms of the index theory:
we extract a geometric content from these objects, 
showing that they define cycles in the representable $K$-homology of the $C_0(X)$-algebra defined by $(\cA,\jmath)_K$.

\

The present paper is organized as follows.

In \S \ref{B} we describe the interplay between presheaves, nets and $C_0(X)$-algebras.

In \S \ref{sec_LUA} we define a functor from nets of $\rC^*$-algebras to $C_0(X)$-algebras.
The $C_0(X)$-algebra $\sA$ associated to a net $(\cA,\jmath)_K$ fulfils a universal 
property, namely the one of lifting any \emph{heteromorphism} from $(\cA,\jmath)_K$ to a 
$C_0(X)$-algebra, defined in a suitable way. This notion is central in our work and is introduced in 
\S \ref{ex.B2.01}.

In \S \ref{s.index} we study nets of Fredholm modules and show that these can be interpreted as
continuous families of Fredholm operators, that is, cycles in the representable $K$-homology $RK^0(\sA)$
defined by Kasparov in \cite{Kas88} (see Appendix \S \ref{KK}).

Finally, in \S \ref{sec.ex} we give some classes of examples, first in the setting of abstract 
$\rC^*$-algebras and then in the one of quantum field theory.

\section{Preliminaries}
\label{A}

In this section we recall some background notions relative to nets of $\rC^*$-algebras
and $C_0(X)$-algebras.

\subsection{Nets of $\rC^*$-algebras}
\label{sec.nets}

We give some recent results on nets of $\rC^*$-algebras. 
Details can be found in \cite{RV12}. 

\paragraph{Posets.} Let $(K,\leq)$ be a partially ordered set (\emph{poset}).  
We shall denote the elements of $K$ by Latin letters $o,a$. We shall write 
$a<o$ to indicate that $a\leq o$ and $a\ne o$.
A poset $K$ is said to be \emph{upward (downward) directed} whenever for any pair  $o_1,o_2\in K$ 
there is $o\in K$ with $o_1,o_2\leq o$ ($o\leq o_1,o_2$). 
If $a\in K$ and $\omega\subseteq K$, then we write $a\leq \omega$ if, and only if, 
$a\leq o$ for any $o\in\omega$. \\ 
\indent The notions of pathwise connectedness and the homotopy equivalence relation can be introduced on an abstract poset $K$
in terms of a simplicial set associated to $K$ \cite{Rob90,RRV09}. This leads to the first homotopy group $\pi^a_1(K)$ of $K$ with respect to a base element $a\in K$. However, for the purposes of the present paper, 
it is not necessary to introduce the explicit definitions in terms of the of the simplicial set. It is enough to recall the following facts:  
if $K$  is pathwise connected, then $\pi_1^a(K)$ 
does not  depend, up to isomorphism, on the choice of the base point $a$; 
this isomorphism class, written $\pi_1(K)$, is the fundamental group of $K$;  
$K$ is said to be {\em simply connected} whenever $K$ is pathwise connected and $\pi_1(K)$ is trivial.  
In particular, if $K$ is either upward  or  downward directed then $K$ is  simply connected.\\ 
\indent In the present paper \emph{when $K$ is a not specified poset we shall always assume that it is pathwise connected}.\\ 
\indent Actually, the posets we shall deal with in the present paper arise as basis $K$, ordered under inclusion, 
of a topological space $X$. In particular if  $X$ is a locally compact, path-connected Hausdorff space, 
then we may choose as the poset 
$K$ a base of $X$, ordered under inclusion, of relatively compact open subsets of $X$. 
Under these conditions, it turns out that the poset $K$ is pathwise connected. 
If, in addition, the elements of $K$ 
can be chosen simply connected, i.e.\ $K$ is a \emph{good base of $X$} (which $X$ amounts to saying that $X$ is a connected, 
locally compact and locally simply connected Hausdorff space) then 
for all $a \in K$ and $x \in a$, there is an \emph{isomorphism} 
\[
\pi_1^a(K) \to \pi^x_1(X)
\]
(see \cite{Ruz05}).
Connected manifolds, that are the class of spaces 
in which we are interested in, fulfill all the above properties.


\paragraph{Nets of $\rC^*$-algebras.} A {\em net of $\rC^*$-algebras} over a poset $K$ is  a pair
$(\cA,\jmath)_K$, where $\cA$ is a collection of unital
$\rC^*$-algebras $\cA_o$, $o \in K$, called the \emph{fibres} of the net, and 
$\jmath$  is a collection of unital $^*$-monomorphisms 
$\jmath_{o'o}:\cA_o\to\cA_{o'}$, for any $o\leq o' $, 
the \emph{inclusion morphisms}, fulfilling  the {\em net relations}
\begin{equation}
\label{net:rel}
\jmath_{o''o'} \circ \jmath_{o'o} = \jmath_{o''o}
 \ , \qquad  o \leq o' \leq o''
\ .
\end{equation}
 If $S \subset K$ and $(\cA,\jmath)_K$ is a net then $(\cA,\jmath)_S$
is itself a net, called the {\em restriction} of $(\cA,\jmath)_K$ to $S$.\\ 
\indent  The easiest example of a net of $\rC^*$-algebras  is that of \emph{the constant net} of $\rC^*$-algebras which is naturally associated to a $\rC^*$-algebra $A$ (that is, $\jmath_{o'o} \equiv id_A$, $\forall o \leq o'$).  
We shall denote this net by the same symbol $A$ as that denoting the constant fibre. 
The \emph{vanishing net}   is a constant net $A$ 
such that $A=0$. \\
\indent A {\em morphism} $\phi : (\cA,\jmath)_K \to (\cB,y)_K$ is a family
of $^*$-morphisms $\phi_o : \cA_o \to \cB_o$, for any $o \in K$,
such that
\begin{equation}
\label{def.netmor}
\phi_{o'} \circ \jmath_{o'o} = y_{o'o} \circ \phi_o
 \ , \qquad   o \leq o'
\ .
\end{equation}
We say that $\phi$ is a {\em morphism into a $\rC^*$-algebra} in the case that 
the codomain net $(\cB,y)_K$ is a constant net $B$. We denote these morphisms 
by $\phi:(\cA,\jmath)_K \to B$. We say that $(\cA,\jmath)_K$ is \emph{trivial} if it is isomorphic 
to a constant net. \smallskip 

%

From the categorical point of view a net of $\rC^*$-algebras is a functor from a poset, considered as a category, 
to the category of unital $\rC^*$-algebras;  morphisms of nets of $\rC^*$-algebras are natural transformations of the corresponding functors.  In what follows we shall deal with nets taking values in other target categories. 
A net of \emph{Banach spaces}  is a pair $(\cH,V)_K$, where $\cH = \{ \cH_o \}$ is a family of Banach spaces and
$V := \{ V_{o'o} \}_{o' \geq o}$
is a family of injective, bounded operators fulfilling the net relations. 
In particular, when each fibre $\cH_o$, $o \in K$, is a Hilbert space and
any operator $V_{o'o}:\cH_o\to \cH_{o'}$, $o'\geq o$, is an isometry we say that $(\cH,V)_K$ is a net of \emph{Hilbert spaces}.


\paragraph{Net bundles.}
When every inclusion morphism  of a net $(\cA,\jmath)_K$  is invertible 
we say that $(\cA,\jmath)_K$ is a {\em $\rC^*$-net bundle}.  Since the morphisms $\jmath_{o'o}$, $o \leq o'$, are invertible, 
it makes sense to consider the inverses $\jmath_{o'o}^{-1}$ and 
to be concise we will write 
$\jmath_{oo'} := \jmath_{o'o}^{-1}$.\\ 
\indent As a consequence of pathwise connectedness of $K$ and of the invertibility of the inclusion morphisms, any
$\rC^*$-net bundle  is, up to isomorphism, of the form $(A,\jmath)_K$ i.e. a net bundle whose 
fibres are all equal to a unique $\rC^*$-algebra $A$, called the \emph{standard fibre},  and whose 
inclusion maps $\jmath_{\tilde o o}$  are \emph{automorphisms} of $A$ (see \cite[\S 3.3]{RV12} for details). \\
\indent In the same way we will talk about Banach and Hilbert net bundles,
where, in the latter case, the inclusion morphisms are unitary operators.

\begin{example}
Let $S \subset \bC$ be a finite set and $X' \subset \bC$ a regular, bounded
open subset such that $S \subset X'$. 
We set $X := X'- S$ and fix a good base $K$ for $X$.
For each $Y \in K$ we consider the Banach space
$\cO^S_Y$
of functions holomorphic in $Y$ having a meromorphic extension to $X$ such that 
the set of poles is contained in $S$.
In this way we get the Banach net bundle $(\cO^S,\jmath)_K$, where 
$\jmath_{Y'Y}$, $Y \subseteq Y'$, is the (non-isometric) operator defined by 
analytic continuation.
\end{example}
%


\paragraph{Representations.} 
A {\em representation} of $(\cA,\jmath)_K$  \emph{into a Hilbert net bundle} $(\cH,U)_K$ is a  morphism 
\begin{equation}
\label{rep}
\pi : (\cA,\jmath)_K \to (\cBH,\ad U)_K \ ,
\end{equation}
where $(\cBH,\ad U)_K$ is the $\rC^*$-net bundle defined by $(\cH,U)_K$ having 
fibres the $\rC^*$-algebra of bounded operators $B(\cH_o)$, $o \in K$, and net structure
defined by adjoint action of the $U_{o'o}$'s. We say that $\pi$ is  \emph{faithful} whenever 
$\pi_o$ is a faithful representation of $\cA_o$ on $\cH_o$ for any $o\in K$, and say that 
$\pi$ is trivial whenever $\pi_o=0$ for any $o$.  \\
\indent Nets of $\rC^*$-algebras are classified according to their representations. A net of $\rC^*$-algebras 
is said to be \emph{non-degenerate} if it admits non-trivial representations, and \emph{injective} 
if it admits faithful representations. Examples of nets exhausting the above classification can be found in \cite{RV12}. See 
\cite{RV12bis} for examples of injective nets remarkable for conformal quantum field theory.   

\begin{remark} 
In the context of  the Algebraic Quantum Field Theory it is customary to consider Hilbert space representations, 
i.e. morphisms from $(\cA,\jmath)_K$ into the single $\rC^*$-algebra $B(H)$. It was only recently that the more general notion 
of representation (\ref{rep}) has been considered \cite{BR09,BFM09}. These two notions agree when the net is defined on 
a simply connected poset. In the non simply connected case there are examples 
of nets having faithful representations but no non-trivial Hilbert space representations 
\cite[Ex.A.9, Ex.5.8(i)]{RV12}. 
\end{remark}

\paragraph{The universal $\rC^*$-algebra and the enveloping net bundle.} 
The \emph{universal $\rC^*$-algebra} is by definition the $\rC^*$-algebra $\vec{A}$ lifting 
any $\rC^*$-algebra morphism of  $(\cA,\jmath)_K$: that is, there is a canonical morphism
$\e :(\cA,\jmath)_K \to \vec{A}$
such that for any
$\phi:(\cA,\jmath)_K \to B$
there is a unique $\phi^\uparrow : \vec{A} \to B$ with $\phi = \phi^\uparrow \circ \e$.
This notion has been introduced by Fredenhagen (\cite{Fre90}),
as a tool for analyzing the superselection structure of nets in conformal field theory. \\
\indent The representation theory of a net is not completely encoded in the universal $\rC^*$-algebra since \emph{only} Hilbert space representations of the net lift  to the universal $\rC^*$-algebra (see previous remark). However it turns out that 
to any net of $\rC^*$-algebras $(\cA,\jmath)_K$ corresponds a $\rC^*$-net bundle $(\bar\cA,\bar\jmath)_K$, the 
\emph{enveloping net bundle} lifting, as above,  a net bundle morphism of  $(\cA,\jmath)_K$. In particular, representations of the net 
are in 1-1 correspondence with those of the enveloping net bundle. \\
\indent The universal $\rC^*$-algebra \emph{forgets} the topology of the poset $K$ (the fundamental group), information which is 
naturally encoded in the enveloping net bundle. Actually when the poset is simply connected these two notions agree: the standard fibre  $\bar{\cA}_o$ of the enveloping net bundle is isomorphic to the universal $\rC^*$-algebra $\vec{\cA}$ (for details, see \cite{RV12}).\\ 
\indent In the present paper we shall use the notion of universal $\rC^*$-algebra to construct the fibres
of our $C_0(X)$-algebras: given $x \in X$ and the downward directed poset 
$\omega_x := \{ Y \in K : x \in Y \}$
(which is simply connected, see \cite[\S 3.2]{RRV09}) we shall define our fibre on $x$
as the universal $\rC^*$-algebra of the restricted net
$(\cA,\jmath)_{\omega_x}$.

\subsection{$C_0(X)$-algebras and the homotopy group}
\label{sec.lc}

Let $X$ be a locally compact Hausdorff space. 
A {\em $C_0(X)$-algebra} is a $\rC^*$-algebra $\sB$ having a nondegenerate 
morphism from $C_0(X)$ to the centre of the multiplier algebra $M \sB$. 
Assuming for simplicity that this morphism is faithful we identify 
elements of $C_0(X)$ with their image in $M \sB$ and write
$fT \in \sB$, $f \in C_0(X)$, $T \in \sB$.
A $\rC^*$-morphism $\eta$ from $\sB$ to the $C_0(X)$-algebra $\sB'$ is said to be a
$C_0(X)$-morphism whenever 
$\eta(fT) = f \eta(T)$,
$\forall f \in C_0(X)$,
$T \in \sB$.
Let $C_x(X)$ denote the ideal of functions vanishing on $x \in X$; then the closed
linear span $C_x(X)\sB$ generated by elements of the type $fT$, $f \in C_x(X)$, $T \in \sB$,
is a closed ideal of $\sB$ and this yields the family of $\rC^*$-quotients
\begin{equation}
\label{def.lc.0}
r^x : \sB \to \sB_x := \sB / \{ C_x(X)\sB \} 
\ , \
x \in X
\ .
\end{equation}
It can be proved that the {\em norm function}
$n_T(x) := \| r^x(T) \|$, $x \in X$,
is upper semicontinuous and vanishes at infinity for every $T \in \sB$ \cite{Nil96,Bla96},
and we say that $\sB$ is a {\em continuous $\rC^*$-bundle} whenever $n_T$
is continuous for any $T \in \sB$.
Defining
\begin{equation}
\label{def.lc.1}
r(T) \ := \ \{ r^x(T) \} \in \prod_x \sB_x
\ \ , \ \
\forall T \in \sB
\ ,
\end{equation}
yields an immersion $\sB \subset \prod_x \sB_x$, and this leads from 
the formalism of $C_0(X)$-algebras to the one originally used by Dixmier 
and Douady (see \cite[Chap.10]{Dix}).

When $X$ is compact analogous considerations hold (without requiring the
property of vanishing at infinity for elements of $\sB$), 
and we use the term {\em $C(X)$-algebra}.


\paragraph{Locally constant bundles.}
Let now $X$ be paracompact and $\sB$ a $C_0(X)$-algebra.
For any $Y \subset X$ we consider the closed ideal $C_0(Y)\sB \subset \sB$.
Given the open cover $\{ X_i \}_{i \in I}$ and a $\rC^*$-algebra $B$, 
an \emph{atlas} of $\sB$ is given by a family $\eta$ of $C_0(X_i)$-isomorphisms
\[
\eta_i : C_0(X_i) \sB \stackrel{\simeq}{\longrightarrow} C_0(X_i) \otimes B 
\ \ , \ \
i \in I
\ ,
\]
and defines $C_0(X_i \cap X_j)$-isomorphisms
\[
\alpha_{ij} \ := \ \eta_i \circ \eta_j^{-1}  : C_0(X_i \cap X_j) \otimes B \to C_0(X_i \cap X_j) \otimes B \ ,
\]
that we regard as maps
$\alpha_{ij} : X_i \cap X_j \to \textbf{aut}B$.

We say that $\sB$ is \emph{locally constant} whenever it has an atlas $\eta$ such that any
$\alpha_{ij}$
is constant on the connected components of $X_i \cap X_j$.
In the sequel we will denote a locally constant $\rC^*$-bundle by $(\sB,\eta)$. 
Let now $(\sB',\eta')$ denote a locally constant $\rC^*$-bundle, 
and $\phi : \sB \to \sB'$ a $C_0(X)$-morphism. Then any
\[
\phi_{i'i} := \eta'_{i'} \circ \phi \circ \eta_i^{-1}
\ : \
C_0( X_i \cap X'_{i'} ) \otimes B \to C_0( X_i \cap X'_{i'} ) \otimes B' 
\]
can be regarded as a continuous map
\begin{equation}
\label{phiij}
\phi_{i'i} : X_i \cap X'_{i'} \to \textbf{hom}(B,B')
\ ,
\end{equation}
where $\textbf{hom}(B,B')$ is the space of morphisms
from $B$ into $B'$ endowed with the pointwise convergence topology. 
We say that $\phi$ is {\em locally constant} whenever each 
$\phi_{i'i}$ is a locally constant map (that is, $\phi_{i'i}$ is constant on any connected component), 
and in this case we write
\[
\phi : (\sB,\eta) \to (\sB',\eta')
\ .
\]
Now in general a $C_0(X)$-morphism is not locally constant, 
so locally constant $C_0(X)$-algebras form a {\em non-full} subcategory of the one of $C_0(X)$-algebras.
The notion of locally constant bundle can be given for generic spaces, 
in particular topological groups and Hilbert (Banach) spaces (\cite[\S I.2]{Kob}); 
this is indeed the common setting of this notion, rather than the one of $\rC^*$-algebras.
\smallskip

\begin{remark}
\label{rem.b0}
Let $A$ be a unital $\rC^*$-algebra. Then applying \cite[Theorem 31]{RRV09} to $G = {\bf aut}A$
we obtain isomorphisms
\[
\textbf{lc}(X,A) \ \simeq \ \textbf{net}(K,A) \ \simeq \ \textbf{dyn}(\pi_1(X),A) \ ,
\]
where:
\begin{itemize}
\item[(i)] $\textbf{lc}(X,A)$ is the set of isomorphism classes of locally constant 
$\rC^*$-bundles over $X$ with fibre (isomorphic to) $A$;
\item[(ii)]  $\textbf{net}(K,A)$ is the set of isomorphism classes of  
$\rC^*$-bundles over $K$ with fibre $A$;
\item[(iii)]  $\textbf{dyn}(\pi_1(X),A)$ is the set of equivalence classes,
under adjoint action of ${\bf aut}A$, of actions of $\pi_1(X)$ on $A$.
\end{itemize}
An analogous result holds in the setting of Hilbert spaces, by replacing $A$
with a Hilbert space $H$ and ${\bf aut}A$ with the unitary group.
\end{remark}

\begin{example}
Let $A$ be a $\rC^*$-algebra with arcwise connected automorphism 
group and $X$ a space with a good base $K$. 
We denote the set of isomorphism classes of locally trivial, continuous 
$\rC^*$-bundles with fibre $A$ by
$\textbf{bun}(X,A)$. 
In accord to \cite[\S 7]{RRV09} and \cite[\S 3.13]{Kar}, in the case of
the $n$-spheres $S^n$, $n \in \bN$, we find
\[
\left\{
\begin{array}{ll}
\textbf{lc}(S^1,A)  \simeq \textbf{aut}^\ad A 
\ , \
\textbf{bun}(S^1,A) \simeq \pi_0(\textbf{aut}A) = \{ 0 \}
\\
\textbf{lc}(S^n,A)   \simeq \{ 0 \}
\ , \
\textbf{bun}(S^n,A) \simeq \pi_{n-1}(\textbf{aut}A)
\ \ , \ \
n \geq 2
\ ,
\end{array}
\right.
\]
where $\textbf{aut}^\ad A$ is the orbit space of $\textbf{aut}A$
under the adjoint action.
\end{example}

\section{Nets and presheaves of $\rC^*$-algebras}
\label{B}

In this section we introduce some objects that will be used to 
connect nets with $C_0(X)$-algebras. 
First, we define presheaves of $\rC^*$-algebras over posets, which
become presheaves in the usual sense in the case of the presheaf
of local sections of a $C_0(X)$-algebra. 
Then we introduce the notion of heteromorphism from a net to a $C_0(X)$-algebra: 
this will give the model of the way in which $(\cA,\jmath)_K$ is connected to $\sA$.\smallskip 

\emph{Unless otherwise stated, we will assume from now on that $X$ is a locally compact, path-connected Hausdorff space and take as poset $K$ a base of $X$, ordered under inclusion, of open relatively compact open subsets of $K$. Recall that this implies that $K$ is pathwise connected poset, see Section \ref{sec.nets}.}

\subsection{Presheaves of $\rC^*$-algebras}
\label{s.B.2}

%

A {\em presheaf} of $\rC^*$-algebras is given a the triple $(\cA,r)^K$, 
where $K$ is a poset, $\cA := \{ \cA_o \}$ a family of $\rC^*$-algebras and a family 
$r:= \{ r_{oo'} : \cA_{o'} \to \cA_o \ , \ o \leq o' \}$,
$^*$-morphisms, called  {\em restriction morphisms},  fulfilling 
{\em the presheaf relations}
\begin{equation}
\label{eq.psheaf}
r_{oo'} \circ r_{o'o''} = r_{oo''}
\ \ , \ \   o \leq o' \leq o''
\ .
\end{equation}
A {\em presheaf morphism} $\phi : (\cA,r)^K \to (\cA',r')^K$ 
is a family of $^*$-morphisms
$\phi = \{ \phi^o \in (\cA_o,\cA'_o) \}$
fulfilling 
$\phi^o \circ r_{oa} = r'_{oa} \circ \phi^{a}$, for any 
$ o \leq a$.

\begin{remark}
When $K$ is a base for the topology of a completely regular space, 
presheaves of $\rC^*$-algebras are, in essence, in one-to-one correspondence with 
$C_0(X)$-algebras: the basic idea is that the given presheaf can be regarded as 
\emph{the} presheaf of local sections of a topological $\rC^*$-bundle, see \cite{AM,Hof}.
\end{remark}


\paragraph{Heteromorphisms.}
Let $(\cA,\jmath)_K$ be a net and $(\cB,r)^K$ a presheaf.
A {\em heteromorphism} from $(\cA,\jmath)_K$ to $(\cB,r)^K$ is given by 
a family $\pi = \{ \pi_o : \cA_o \to \cB_o \}$ of $\rC^*$-morphisms 
fulfilling
\begin{equation}
\label{eq.B2.01}
r_{oo'} \circ \pi_{o'} \circ \jmath_{o'o} = \pi_o
\ \ , \ \  o \leq o'
\ .
\end{equation}
In the sequel, heteromorphisms  shall be denoted by
\[
\pi : (\cA,\jmath)_K \nrightarrow (\cB,r)^K \ ,
\]
emphasizing the "wrong-way functoriality" of the notion.
Note that regarding $(\cA,\jmath)_K$ as a functor from $K$ into the category of $\rC^*$-algebras
and $(\cB,r)^K$ as a cofunctor we find that $\pi$ can be interpreted as a natural transformation. 
When each $r_{oo'}$ is invertible we have the $\rC^*$-net bundle
$(\cB,r^{-1})_K$ and we can regard $\pi$ as a morphism of nets;
thus the idea of heteromorphism into a presheaf is a natural generalization
of the one of morphism into a $\rC^*$-net bundle. Finally, we point out that the above definitions easily generalize to the case where the algebras of the involved nets are not unital.


\subsection{The presheaf defined by a $C_0(X)$-algebra}
\label{ex.B2.01}

In this section we consider the presheaf of $\rC^*$-algebras associated to a given $C_0(X)$-algebra:
this will allow us to define heteromorphisms from nets to $C_0(X)$-algebras.


\paragraph{Restriction morphisms.}
Given the $C_0(X)$-algebra $\sB$, for any open set $U \subset X$ we consider the closed ideal
$C_0(U) \sB$
generated by elements of the type $fv$, $f \in C_0(U)$, $v \in \sB$.
Let us consider the $\rC^*$-algebras
$S_Y \sB := \sB / C_0(X - \ovl Y) \sB$, $Y \in K$,
and the corresponding quotients
\begin{equation}
\label{Cb:2}
r_Y: \sB \to S_Y \sB \ \ , \ \ Y \in K \ .
\end{equation}
\begin{lemma}
\label{lem.Cb}
Let $t,s \in \sB$ and $Y \in K$. Then
$r_Y(t) = r_Y(s)$ if, and only if, 
$r^x(t)  =  r^x(s)$ for any  $x \in Y$.

\end{lemma}

\begin{proof}
We have $r_Y(s) = r_Y(t)$ if and only if
$t-s = fv$
for some 
$f \in C_0(X - \ovl Y)$, $v \in \sB$,
and this implies
$r^x(t) - r^x(s) = f(x) r^x(v) = 0$, $\forall x \in \ovl Y$,
that is,
\[
r^x(t) = r^x(s) \ , \ \forall x \in Y \ .
\]
Conversely, let $t,s \in \sB$ fulfilling the above relations. 
To prove the Lemma it suffices to verify that $v := t-s$ belongs to $C_0(X- \ovl Y)\sB$.
To this end, we note that 
$\| v \| = \sup_{x \in X-Y}\| r^x(v) \|$
and that, by upper-semicontinuity of $n_v$, for any $\eps > 0$ 
the set $O_\eps:=\{ x\in X \,| \, \|r^x(v) \| < \eps\}\supset \ovl Y$ is open.  
By a standard application of the  Urysohn Lemma we get: a continuous function 
$f_\eps:X\to[0,1]$ and a compact set $W_\eps$,  with $O_\eps\supset W_\eps\supset \ovl Y$, 
such that $f_\eps =0$ on $\ovl Y$ and $f_\eps=1$ on $X-W_\eps$; an approximate unit  $\{g_\lambda\}$ of $C_0(X)$, where the index $\lambda$ varies  in the set of compact subsets of $X$, such that  $g_\lambda(x)=1$ for $x\in \lambda$. So, taking $\lambda$ such that $W_\eps\subseteq \lambda$ and $\| v - g_\lambda v\| < \eps$, we get 
\[
\| v - f_\eps g_\lambda v \| \leq 
\| v - g_\lambda v \| + \| g_\lambda v - f_\eps g_\lambda v \| =
\| v - g_\lambda v \| + \sup_{x \in W_\eps-\ovl Y} \| r^x(v) - f_\eps(x) r^x(v) \| < 2 \eps \ , 
\]
and this, since $f_\eps g_\lambda \in C_0(X-\ovl Y)$ for any $\lambda$ and $\eps$,  completes the proof.  
\end{proof}


\paragraph{Pullbacks of local sections.}
Local sections of $\sB$ can be smoothed by continuous  functions in such a way to define
elements of $\sB$ itself, as follows:
\begin{equation}
\label{Cb:1}
C_0(Y) \times S_Y \sB \to \sB
\ \ , \ \
f , r_Y(t) \mapsto f \tr r_Y(t) := ft 
\ .
\end{equation}
This operation is well-defined because $ft = ft'$ for any $t' \in \sB$ such that $r_Y(t) = r_Y(t')$.

\begin{corollary}
\label{cor.zzz}
Let $w,w' \in S_Y \sB$ such that $f \tr w = f \tr w'$ for all $f \in C_0(Y)$.
Then $w=w'$.
\end{corollary}

\begin{proof}
By the previous Lemma it suffices to check that $r^x(w) = r^x(w')$ for all $x \in Y$.
To this end, given $x \in Y$ we take $f \in C_0(Y)$ such that $f(x) = 1$, so
$r^x(f \tr w) = f(x)r^x(w) = r^x(w)$,
and, in the same way, $r^x(f \tr w') = r^x(w')$.
Since by hypothesis $r^x(f \tr w) = r^x(f \tr w')$ we find $r^x(w) = r^x(w')$, 
so $w=w'$ as desired.
\end{proof}


\paragraph{The presheaf structure.}
Since the inclusion of ideals $C_0(X - \ovl{Y'})\sB \subseteq C_0(X - \ovl Y)\sB$ holds for any  $Y\subseteq Y'$,   
we have a $^*$-morphism $r_{YY'} : S_{Y'} \sB \to S_Y \sB$ defined by 
\begin{equation}
\label{Cb:3} 
r_{YY'}\circ r_{Y'} := r_Y  \ , \ \   Y\subseteq Y' \ , 
\end{equation}
which obviously fulfils the presheaf relations
$r_{YY''} = r_{YY'} \circ r_{Y'Y''}$, for any  $Y \subseteq Y' \subseteq Y''$,
and yield the presheaf \smash{$(S \sB,r)^K$}. 
Furthermore, for the same reason as in the previous definition,  for any $x\in Y$ we have a $^*$-morphism 
$r^x_Y : S_Y \sB \to \sB_x$  defined as 
\begin{equation}
\label{Cb:2a}
r^x_Y\circ r_Y := r^x  \ \ ,  \ \   x\in Y \ . 
\end{equation}
Concerning the functoriality of the above constructions, any $C_0(X)$-morphism 
$\phi : \sB \to \sB'$  
induces in an obvious way the morphism
\begin{equation}
\label{Cb:6}
\phi^s : (S\sB, r)^K\to (S\sB', r')^K 
\ \ , \ \ 
\phi^{s,Y}\circ r_Y := r'_Y\circ \phi  \ \ , \ \ Y\in K \  .    
\end{equation}
This leads to a functor 
from the category of $C_0(X)$-algebras to the category of presheaves of $\rC^*$-algebras.
About the functoriality of (\ref{Cb:1}) we note that
\begin{equation}
\label{Cb:1a}
\phi(f \tr r_Y(t)) \ \stackrel{ (\ref{Cb:1}) }{=} \ 
\phi(ft) \ = \
f \phi(t) \ \stackrel{ (\ref{Cb:1}) }{=} \
f \tr r'_Y(\phi(t)) \ \stackrel{ (\ref{Cb:6}) }{=} \
f \tr \phi^{s,Y}(r_Y(t)) \in \sB' \ .
\end{equation}


\paragraph{Morphisms from nets to $C_0(X)$-algebras.}
We conclude by giving the notion of \emph{heteromorphism from a net $(\cA,\jmath)_K$ to a $C_0(X)$-algebra $\sB$},
that is, a heteromorphism
\smash{$\phi : (\cA,\jmath)_K \nrightarrow(S \sB,r)^K$}.
The geometric content of $\phi$ is that any $t \in \cA_Y$, $Y \in K$, defines the local section 
$\phi_Y(t) \in S_Y \sB$
that can be extended, using the net structure of $(\cA,\jmath)_K$, to the local section 
$\phi_{Y'} \circ \jmath_{Y'Y}(t)$.  
To be concise, in the sequel $\phi$ will be denoted by
\begin{equation}
\label{eq.comorph}
\phi : (\cA,\jmath)_K \nrightarrow  \sB \ .
\end{equation}
Compositions of (hetero)morphisms between nets and $C_0(X)$-algebras shall be
denoted,  coherently with the above notation, as follows:
let $(\cA_0,\jmath_0)_K$, $(\cA,\jmath)_K$ be nets of $\rC^*$-algebras, 
let $\sB$, $\sB'$ be  $C_0(X)$-algebras,  and
\[
\phi : (\cA,\jmath)_K \nrightarrow  \sB
\ \ , \ \
\eta : (\cA_0,\jmath_0)_K \to (\cA,\jmath)_K
\ \ , \ \
\psi : \sB \to \sB'
\]
(hetero)morphisms. Then we define 
\begin{equation}
\phi \circ \eta : (\cA_0,\jmath_0)_K \nrightarrow  \sB \ \ , \ \  (\phi\circ\eta)_Y :=   \phi_Y\circ\eta_Y \ \ , \ \ Y\in K \ , 
\end{equation}
and  using (\ref{Cb:6})
\begin{equation}
\label{eq.comp}
\psi^s \circ \phi : (\cA,\jmath)_K \nrightarrow  \sB' \ \ , \ \  (\psi^s \circ \phi)_Y := \psi^{s,Y} \circ \phi_Y \ \ , \ \ Y\in K 
\ .
\end{equation}

\section{From nets of $\rC^*$-algebras to $C_0(X)$-algebras}
\label{sec_LUA}

In this section we prove our main result: we associate  a $C_0(X)$-algebra
to a net of $\rC^*$-algebras and prove  that this assignment is a functor. 
The assignment is not trivial when the net is non-degenerate and, in particular, is \emph{faithful} in a suitable sense when the net is injective. 
Furthermore, we shall see that sections of the net induce multipliers of the
corresponding $C_0(X)$-algebra. We conclude by drawing  some consequences 
on the representation theory of a net. \\
\indent It is worth recalling that we are assuming 
that $X$ is a locally compact path-connected Hausdorff space
and that $K$ is a base of $X$, ordered under inclusion, of relatively compact open subsets of $X$
(see the begin of Section \ref{B}).

\subsection{The $C_0(X)$-algebra of a net}
\label{DA}
In the following lines we construct the $\rC_0(X)$-algebra associated to a net of $\rC^*$-algebras and show 
that this assignment satisfies a universal property  yielding  a functorial structure.\smallskip

Let  $(\cA,\jmath)_K$ be a net of $\rC^*$-algebras. For any  $x\in X$, consider  the set
$\omega_x := \{ Y \in K : x \in Y \}$
and the net $(\cA,\jmath)_{\omega_x}$ obtained by restricting $(\cA,\jmath)_K$ to $\omega_x$. 
%
%
Denoting the universal $\rC^*$-algebra of $(\cA,\jmath)_{\omega_x}$ by $\cA_x$ (see \S \ref{sec.nets}), 
the canonical morphism  
$\e^x  :(\cA,\jmath)_{\omega_x}\to \cA_x$ 
satisfies the relations
\begin{equation}
\label{Da:00}
\e^x_{Y'}\circ \jmath_{Y'Y}= \e^x_Y \  , \qquad  Y\subseteq Y' \ , Y,Y'\in\omega_x \ .
\end{equation}
The basic idea is that the $C_0(X)$-algebra associated with the net $(\cA,\jmath)_K$ 
is generated by a suitable set of vector fields of  the product $\rC^*$-algebra $\cA^X:=\prod_{x\in X} \cA_{x}$ 
(i.e., the $\rC^*$-algebra of vector fields $t:x\mapsto t_x\in \cA_x$ endowed with the $*$-algebraic structure 
defined componentwise and the sup-norm $\|t\|:= \sup_x\|t_x\|_x$). \smallskip 

To begin with, note that for any $Y\in K$ there is a morphism $\widehat{\e}_Y:\cA_Y\to \cA^X$ 
defined by 
\begin{equation}
\label{Da:1}
\widehat{\e}_{Y}(t)_x 
:= \left\{
\begin{array}{ll}
 \e^x_Y(t)  \ , &  x\in Y \ , \\
0  \ , & otherwise  \ ,
\end{array}
\right.
\end{equation} 
for any $t\in\cA_Y$.
The $^*$-algebra generated by vector fields $\widehat{\e}_Y(t)$, as $Y$ and $t$ vary in $K$ and $\cA_Y$ 
respectively, is not closed under pointwise multiplication of continuous functions. So we smooth these fields as follows 
\[
( f\,\widehat{\e}_Y(t))_x:= f(x)\, \widehat{\e}_Y(t)_x \ ,\qquad Y\in K \ , \ f\in C_0(Y)  \ , \ t\in\cA_Y \ , 
\]
and define $\cA_*$ to be the $^*$-algebra generated by the fields $f\,\widehat{\e}_Y(t)$ 
as $Y$, $f$ and $t$ vary in $K$, $C_0(Y)$ and $\cA_Y$ respectively. 
A generic element $\hat t$ has  the form 
\[
\hat t = \sum_{s\in S} \prod_{i\in I_s} f_{i} \,\widehat{\e}_{Y_{i}}(t_{i}) \ , \qquad    \ Y_{i}\in K \ ,  \ f_{i}\in C_0(Y_{i}) \ , \ t_{i}\in \cA_{Y_{i}} \ , 
\]
where $S$ and $I_s$, for $s\in S$, are sets of indices having a finite cardinality. Note in particular that any $\hat t$ has a  compact support since it is contained 
in $\cup_{s\in S}\cup_{i\in I_s} Y_i$ which is  relatively compact. This implies that 
$\cA_*$ is a non-degenerate $C_0(X)$-module under pointwise multiplication, in fact for any approximate unit $\{f_\lambda\}$ of $C_c(X)$ the identity  
$f_\lambda\hat t =\hat t$ holds eventually.  Finally, we define the $\rC^*$-algebra
\begin{equation}
\label{Da:2}
\sA \ := \ \cA^{-\|\, \|}_* \subset \cA^X \ ;
\end{equation}
in words, $\sA$ is the closure of $\cA_*$ under the sup-norm in $\cA^X$. \\
\indent We shall see soon that $\sA$ is indeed a $C_0(X)$-algebra and yields an interpretation of $(\cA,\jmath)_K$ 
as a net whose elements are local sections of $\sA$ that can be extended by means of $\jmath$. 
Before showing this, we need a preliminary observation.
\begin{remark}
We point out that the collection $\{\widehat {\e}_Y \ , \ Y\in K\}$ is not a morphism from the net $(\cA,\jmath)_K$ to the $\rC^*$-algebra $\cA_X$. In fact the relation 
$\widehat{\e}_Y\circ \jmath_{YY'} =\widehat{\e}_{Y'}$ does not hold in general but only in restriction to $Y'$ i.e.\ 
\[
\widehat{\e}_{Y'}(t)_x  = \widehat{\e}_{Y}(\jmath_{YY'}(t))_x \ ,  \qquad x\in Y'\subseteq Y \ , 
\]
as can be easily seen by  (\ref{Da:00}).  This and  the relations  (\ref{Da:00})  imply   
\begin{equation}
\label{Da:1a}
f \,\widehat{\e}_{Y}(t) \ = \ f \,\widehat{\e}_{Y'} (\jmath_{Y'Y}(t)) \ , \qquad  Y\subseteq Y', \ f\in C_0(Y), \ t\in\cA_Y \ .   
\end{equation} 
\end{remark}
\begin{theorem}
\label{thm.b2}
Let $(\cA,\jmath)_K$ a net of $\rC^*$-algebras. Then $\sA$ is a $C_0(X)$-algebra 
endowed with a canonical heteromorphism
\[
 \tau : (\cA,\jmath)_K \nrightarrow  \sA \ .
\]
\end{theorem}

\begin{proof}
$\sA$ is a $\rC^*$-algebra, so we verify that it is a $C_0(X)$-algebra. Note that,  
for any $f \in C_0(X)$, $\hat t \in \cA_*$ it turns out
$\| f \hat t \| \ = \ 
\sup_{x\in X} \| f(x) \, \hat t_x \|_x \ \leq \ 
 \| f \|_\infty \| \hat t \|
$,   
so $C_0(X)$ acts by bounded linear operators on $\sA$ and we can extend by continuity
the $C_0(X)$-action on $\sA$. 
By density, it is easily seen that $C_0(X)$ acts on $\sA$ by central multipliers.
To prove that this action is non-degenerate we take $T \in \sA$, pick $\eps > 0$
and $\wa t_\eps \in \cA_*$ such that $\| T - \wa t_\eps \| < \eps$.
Given an approximate unit $\{ f_\lambda \} \subset C_c(X)$ there is, as observed before, $\lambda_\eps$
such that $\wa t_\eps = f_\lambda \wa t_\eps$ for all $\lambda > \lambda_\eps$, so
\[
\| T - f_\lambda T\| \leq
\| T - \wa t_\eps \| + \| f_\lambda(T - \wa t_\eps) \| \leq
2 \eps \ . 
\]
This proves that the $C_0(X)$-action is non-degenerate, so $\sA$ is a $C_0(X)$-algebra. \smallskip

We now define the canonical morphism $\tau$. Given $Y\in K$, we consider $U\in K$ such that $\overline{Y}\subset U$ and  
a \emph{plateau} $g$ of $C_0(U)$ on $Y$, i.e. a function $g$ of $C_0(U)$ such that $g=1$ on $Y$. Define 
\[
\tau_Y(t) \ := \ r_{Y}(g \ \widehat{\e}_{U}(\jmath_{UY}(t)) \  ) \ , \qquad   t \in \cA_Y \ , 
\]
where we used the quotient
$r_Y : \sA \to \sA / \{ C_0(X-\overline{Y}) \sA \}$.  
We prove that this definition is well posed. 
We first note that we have independence of the choice of the plateau on $Y$, because if 
$h \in C_0(U)$ is a plateau on $Y$, then   
\[
g \,\widehat{\e}_{U}(\jmath_{UY}(t))- h \,\widehat{\e}_{U}(\jmath_{UY}(t)) = 
(g-h)\widehat{\e}_{U}(\jmath_{UY}(t))\in C_0(X-\overline{Y})\sA = \ker r_Y \ .
\]
To prove that $\tau_Y(t)$ is independent of the choice of $U$ we make the following remark.
Consider $V \in K$ such that $\overline{Y}\subset V\subset U$ and observe that if $\tilde g\in C_0(V)$ then  
\begin{equation}
\label{eq.b2}
\tilde g\,\widehat{\e}_{U}(\jmath_{UY}(t)) \ = \ 
\tilde g\,\widehat{\e}_{U}(\jmath_{UV} \circ \jmath_{VY}(t)) \ \stackrel{(\ref{Da:1a})}{=} \ 
\tilde g \,\widehat{\e}_{V}(\jmath_{VY}(t)) \ .
\end{equation}

We now can prove the independence of $U$. 
Given $U' \supseteq Y$ and $g' \in C_0(U')$ such that $g' = 1$ on $Y$, we note that 
we can always find $V$ such that $\overline{Y}\subset V$, $\overline{V} \subset U \cap U'$,
and $\tilde g \in C_0(V)$ such that $\tilde g = 1$ on $Y$. 
We have
\[
\begin{array}{ll}
g \,\widehat{\e}_{U}(\jmath_{UY}(t)) - g' \,\widehat{\e}_{U'}(\jmath_{U'Y}(t)) & = \
\tilde g \,\widehat{\e}_{U}(\jmath_{UY}(t)) - \tilde g \,\widehat{\e}_{U'}(\jmath_{U'Y}(t)) + F  \ \stackrel{(\ref{eq.b2})}{=} \\ & = \
\tilde g \,\widehat{\e}_{V}(\jmath_{VY}(t)) - \tilde g \,\widehat{\e}_{V}(\jmath_{VY}(t)) + F \ = \ F \ ,
\end{array}
\]
where 
$F := (g-\tilde g) \widehat{\e}_{U}(\jmath_{UY}(t)) + (g'-\tilde g) \widehat{\e}_{U'}(\jmath_{U'Y}(t))$
is in $\ker r_Y$ by plateau independence.
This proves that the definition of $\tau_Y : \cA_Y \to S_Y \sA$ is well posed and clearly it is a linear map.
To prove that 
$\tau_Y$ is a $^*$-morphism we note that for any $t,s\in\cA_Y$ we have  
\[
\tau_Y(t^*s)= r_{Y}(g \ \widehat{\e}_{U}(\jmath_{UY}(t^*s)) = 
\ r_{Y}\big(\sqrt{g} \ \widehat{\e}_{U}(\jmath_{UY}(t))^* \sqrt{g} \,\widehat{\e}_{U}(\jmath_{UY}(s))\big)  =
\tau_Y(t)^* \tau_Y(s) \ , 
\]
because of the independence of the choice of the plateau, since $\sqrt{g}$ is a plateau of $C_0(U)$ on $Y$. 
Finally, if $Y \subset Y'$  and $t\in\cA_{Y}$ then 
\[
r_{YY'}\circ \tau_{Y'} (\jmath_{Y'Y}(t)) \ = \   
r_{YY'}\circ r_{Y'}(g \,\widehat{\e}_{U}(\jmath_{UY'}\circ \jmath_{Y'Y}(t)) \ \stackrel{(\ref{Cb:3})}{=} \
r_{Y}(g \,\widehat{\e}_{U}(\jmath_{UY}(t)) \ = \
\tau_{Y}(t) \ ,
\]
because $U$ and $g$ satisfies all the properties for the definition of $\tau_{Y}$. 
This proves that $\tau$ is an heteromorphism as desired.
\end{proof}
\begin{remark}
By the above construction it seems natural to identify the fibres 
$\sA_x := \sA / \{ C_x(X)\sA \}$, $x \in X$,
with the $\rC^*$-algebras $\cA_{x}$. Indeed, the fact that $\cA_{x}$ is isomorphic to $\sA_x$
shall be proved in the sequel, see Prop.\ref{prop.uni}.
\end{remark} 
To describe explicitly the canonical heteromorphism we note that, for all $Y \in K$,
$f \in C_0(Y)$ and $t \in \cA_Y$,
\begin{equation}
\label{Da:10}
f \tr \tau_Y(t) \ = \ 
fg \, \widehat{\e}_U(\jmath_{UY}(t)) \ \stackrel{ g|_Y = 1 }{=} \
f \widehat{\e}_U(\jmath_{UY}(t)) \ = \
f \widehat{\e}_Y(t) \ .
\end{equation}
This also shows that $\tau$ has \emph{dense image}, in the sense that the set
\[
\{ f \tr \tau_Y(t)  \ : \  Y \in K \ , \ t \in \cA_Y \ , \ f \in C_0(Y)  \}
\]
generates $\sA$ as a $\rC^*$-algebra.
We can now show a universal property of $\sA$.

\begin{theorem}
\label{thm.uni}
Let $(\cA,\jmath)_K$ be a net of $\rC^*$-algebras. Then $\sA$ fulfils 
the following universal property:
for any heteromorphism $\phi : (\cA,\jmath)_K \nrightarrow \sB$ into a $C_0(X)$-algebra $\sB$,
there is a unique $C_0(X)$-morphism $\bar \phi : \sA \to \sB$ such that 
$\phi = {\bar \phi}^s \circ \tau$.
\end{theorem}

\begin{proof}
By hypothesis, for any $Y \in K$ there is a morphism 
$\phi_Y:\cA_Y\to S_Y \sB$ 
such that 
$\phi_{Y} = r_{YY'}\circ \phi_{Y'}\circ \jmath_{Y'Y}$ for any inclusion  $Y\subseteq Y'$,
where $r_{YY'}$ is defined in (\ref{Cb:3}).
Let now $x\in X$; consider the morphism $\phi^x_Y : \cA_Y \to \sB_x$ defined by
\[
\phi^x_Y := r^x_Y \circ \phi_Y \ , \qquad  Y \in\omega_x \ , 
\]
where  $r^{x}_{Y}$ is the morphism defined by (\ref{Cb:2a}).  The defining relations (\ref{Cb:3}) and (\ref{Cb:2a}) imply  that $r^x_{Y'} = r^{x}_{Y} \circ r_{YY'}$ for any $Y,Y'\in\omega_x$, $Y \subseteq Y'$, 
so
\[
\phi^x_{Y'} \circ \jmath_{Y'Y} \ = \
r^x_{Y'} \circ \phi_{Y'} \circ \jmath_{Y'Y} \ = \
r^{x}_{Y} \circ r_{YY'} \circ \phi_{Y'} \circ \jmath_{Y'Y} \ = \
r^{x}_{Y} \circ ( r_{YY'} \circ \phi_{Y'} \circ \jmath_{Y'Y}) \ = \
r^x_Y \circ \phi_Y \ = \
\phi^x_Y \ ,
\]
that is, we have a morphism from $(\cA,\jmath)_{\omega_x}$ into $\sB_x$,
inducing by universality a $\rC^*$-morphism 
$\phi^x : \cA_{x}  \to  \sB_x$, for any $x \in X$, satisfying the relations
\begin{equation}
\label{cx.c1a}
\phi^x\circ \e^x_Y = \phi^x_Y = r^x_Y \circ \phi_Y \ , \qquad Y\in\omega_x \ .
\end{equation}
We now consider the $^*$-morphism 
\begin{equation}
\label{cx.c1}
\phi_* : \cA_* \to \prod_x \sB_x
\ \ , \ \
\phi_*(t) := \{ \phi^x(t_x) \}_{x \in X}
\ , \qquad  t \in \cA_* \ ,
\end{equation}
and show that $\phi_*(\cA_*) \subseteq r(\sB)$, where $r$ is defined in (\ref{def.lc.1}). 
It is enough to prove that for the generators of $\cA_*$, 
i.e. for the fields 
$f \,\widehat{\e}_Y(t)$, with $Y \in K$, $f \in C_0(Y)$, $t \in \cA_Y$.
To this end, we note that by construction
$f \tr \phi_Y(t) \in \sB$,
and compute, for all $x \in X$,
\[
r^x (f \tr \phi_Y(t)) \ = \
f(x) \cdot r^x_Y \circ \phi_Y(t) \ = \
f(x) \cdot \phi^x\circ \e^x_Y(t) \ = \
\phi^x( \{ f \widehat{\e}_Y(t) \}_x ) \ ,
\]
that is, 
$r(\sB) \ni r(f \tr \phi_Y(t)) = \phi_*(f \,\widehat{\e}_Y(t))$
as desired. Since 
$\| \phi_*(t) \| = \sup_x \| \phi^x(t_x) \| \leq \sup_x \| t_x \| = \| t \|$ for all $t \in \cA_*$,
we can extend $\phi_*$ to $\sA$ and define $\bar \phi(t) := r^{-1} \circ \phi_*(t)$.
So, the above equalities read as
\begin{equation}
\label{eq.uni}
f \tr \phi_Y(t)     \ = \    \bar{\phi}(f \,\widehat{\e}_Y(t)) \ .
\end{equation}
To prove that $\bar{\phi}$ fulfils the desired universal property we note that
\[
f \tr \phi_Y(t) \ = \ 
\bar{\phi}(f \widehat{\e}_Y(t)) \ \stackrel{ (\ref{Da:10}) }{=} \ 
\bar{\phi}(f \tr \tau_Y(t)) \ \stackrel{ (\ref{Cb:1a}) }{=} \ 
f \tr \bar{\phi}^{s,Y}(\tau_Y(t))
\]
for any $f \in C_0(Y)$, $t \in \cA_Y$, and the proof follows by Cor.\ref{cor.zzz}.
Finally, to prove uniqueness we suppose that there is $\eta : \sA \to \sB$ such that
$\phi = \eta^s \circ \tau$;
in explicit terms, this means that
$\phi_Y(t) = \eta^{s,Y} \circ \tau_Y(t)$, $\forall Y \in K$, $t \in \cA_K$,
but this implies
$f \tr \bar{\phi}^{s,Y}(\tau_Y(t))  = 
f \tr \phi_Y(t)  = 
f \tr \eta^{s,Y}(\tau_Y(t))$, for any  $f \in C_0(Y)$. So, applying (\ref{Cb:1a}) we conclude that
\[
\bar{\phi}( f \tr \tau_Y(t)) \ = \ 
f \tr \bar{\phi}^{s,Y}(\tau_Y(t)) \ = \
f  \tr \eta^{s,Y}(\tau_Y(t))  \ = \ 
\eta ( f \tr \tau_Y(t)) \ ,
\]
for all $Y \in K$, $f \in C_0(Y)$, $t \in \cA_Y$, and uniqueness  follows by density of $\tau$.
\end{proof}

Applying the previous theorem to $\tau : (\cA,\jmath)_K \nrightarrow\sA$ we find 
$\tau = {\bar{\tau}}^s \circ \tau$
and this implies, by density of $\tau$, that $\bar{\tau}$ is the identity of $\sA$. We now draw on a consequence of this fact. 
Let $e_x:\sA\to\cA_{x}$ be the \emph{evaluation} morphism which is defined on the fields $t$ generating  $\sA$ as  
\[
e_x(t):= t_x \ , \qquad t\in\cA_* \ , 
\]
and extended by continuity on all $\sA$. 
The evaluation morphism is surjective: 
in fact, according to subsection \ref{DA} 
the image of the canonical morphism $e^x$ is dense in $\cA_x$,  
thus it is enough to observe that for any $Y$ containing $x$
we have that $e_x(\widehat e_Y(\cA_Y))=e^x_Y(\cA_Y)$.  
\begin{proposition}
\label{prop.uni}
Let $(\cA,\jmath)_K$  be a net of $\rC^*$-algebras. Then  the equation  (\ref{cx.c1}) applied to 
the canonical heteromorphism  $\tau:(\cA,\jmath)_K\nrightarrow \sA$ defines 
an isomorphism $\tau^x:\cA_{x}\to \sA_x$ satisfying the equation  $r^x=\tau^x\circ e_x$ for any $x\in X$. 
\end{proposition}
\begin{proof}
As observed $id=\bar\tau=r^{-1}\circ \tau_*$ where $\tau_*$ is defined by (\ref{cx.c1}). Since $r^{-1}$ is an isomorphism 
we have that $r=\tau_*$. This equality amounts to saying that  $ r^x=\tau^x\circ e_x$ for any  $x \in X$,  
and implies, since $r^x$ is surjective, that $\tau^x$ is surjective.
We now prove that $\tau^x$ is injective on  $\cA_{x}$. 
For, assume that $r^x(t)=0$ for $t\in\sA$. This is equivalent to saying that 
$t= f\, t'$ where $t'\in\sA$ and $f\in C_x(X)$ (so $f(x)=0$). 
Let now $t'_\lambda$ be a net of fields in $\cA_*$ converging to 
$t'$. Then $f\, t'_\lambda\in\cA_*$ and converges to $f\,t'=t$, so 
$(f\,t'_\lambda)_x=f(x) t'_{\lambda,x}=0$.  Hence 
\[
e_x(t)= e_x(f\, t')=\lim_\lambda e_x(f\, t'_\lambda) = \lim_\lambda f(x)\, t'_{\lambda,x} = 0 \ .    
\]
This, and surjectivity of $e_x$ complete the proof.
\end{proof}

The universality property proved in the previous theorem is also useful to compute $\sA$: 
in fact, to prove that a $C_0(X)$-algebra $\sA'$ is isomorphic to
$\sA$ it suffices to verify that it pulls back any heteromorphism
$\phi : (\cA,\jmath)_K \nrightarrow \sB$.

\begin{example}[Including the case of nets on Minkowski spacetime]
Let $K$ be a base for $X$ \emph{directed} under inclusion (so that $X$ is 
simply connected) and $(\cA,\jmath)_K$ a net of $\rC^*$-algebras over $K$. 
Then the universal algebra $\vec{\cA}$ is the inductive limit
$\lim_Y(\cA_Y,\jmath_{Y'Y})$,
so we regard each $\cA_Y$ as a $\rC^*$-subalgebra of $\vec{\cA}$ and 
any $\jmath_{Y'Y}$, $Y \subseteq Y'$, as the inclusion morphism.
Let now $y \in X$ and $Y \in \omega_y$ (i.e. $y \in Y$); 
if $x \in X$ then there is $Y'$ such that $x,y \in Y' \supseteq Y$ so, 
repeating the argument for all $Y \in \omega_y$, we conclude
$\cA_y \subseteq \cA_{x}$.
Reversing the argument, we conclude that $\cA_y = \cA_{x}$ for all $x,y \in X$,
and consequently $\vec{\cA} = \cA_{x}$ for all $x \in X$.
Since we regard any $\jmath_{Y'Y}$ as the inclusion, we have that any 
$\e^x_Y$, $x \in Y$,
is the inclusion $\cA_Y \subseteq \vec{A}$, so $\wa t_x = t$ for all $t \in \cA_Y$ and $x \in Y$.
Let now $\sA' := C_0(X) \otimes \vec{\cA}$ and 
$\phi : (\cA,\jmath)_K \nrightarrow  \sB$
a heteromorphism. Then we set
\[
\bar \phi : \sA' \to \sB \ \ , \ \ \bar \phi (f \otimes t) := f \tr \phi_Y(t)
\ \ , \ \
f \in C_0(Y) \ , \ t \in \cA_Y \subseteq \vec{A}
\ .
\]
It is easily verified that $\phi = {\bar \phi}^s \circ \tau$, 
and this implies $\sA \simeq C_0(X) \otimes \vec{\cA}$.
\end{example}

\begin{proposition}
\label{prop.b2}
If
$\phi : (\cA_1,\jmath_1)_K \to (\cA_2,\jmath_2)_K$
is a morphism of nets of $\rC^*$-algebras then there is a $C_0(X)$-morphism 
$\phi^\tau : \sA_1 \to \sA_2$,
and this makes $\{ (\cA,\jmath)_K \to \sA \}$ a functor.
\end{proposition}

\begin{proof}
Let $\tau_k : (\cA_k,\jmath_k)_K \nrightarrow  \sA_k$, $k=1,2$, 
denote the canonical embeddings. Then we have the morphism
\[
\phi' := \tau_2 \circ \phi : (\cA_1,\jmath_1)_K \nrightarrow  \sA_2 \ ,
\]
inducing, by the previous Theorem, the $C_0(X)$-morphism 
$\phi^\tau : \sA_1 \to \sA_2$,
$\phi^\tau := \bar \phi'$.
Again by the previous theorem, we find
\begin{equation}
\label{eq.psf2}
(\phi^\tau)^s \circ \tau_1 =  \phi' = \tau_2 \circ \phi \ ,
\end{equation}
%
%
From this equality and the density of the image of the canonical heteromorphism  (\ref{Da:10}) we easily deduce the functoriality property
$\{ \phi \circ \chi \}^\tau = \phi^\tau \circ \chi^\tau$.
\end{proof}

\

In the next result we determine how the fields generating $\sA$ transform under the action of net morphisms:
\begin{corollary}
\label{cor.b3}
Let $\phi : (\cA,\jmath)_K \to (\cA',\jmath')_K$ be a morphism of nets of $\rC^*$-algebras. Then 
\[
\phi^\tau ( f \, \widehat{\e}_Y(t) )   \ = \   f \, \widehat{\e'}_Y( \phi_Y(t) ) \ ,  \qquad Y \in K, \ f\in C_0(Y), \  t \in \cA_Y.
\]
\end{corollary}

\begin{proof}
By (\ref{eq.psf2}) we have 
$\tau'_Y \circ \phi_Y(t) = \phi^{\tau,Y} \circ \tau_Y (t) \in S_Y \sA'$ 
for all $t \in \cA_Y$, $Y \in K$. So applying  (\ref{Cb:1a})  we find that 
$f \tr \tau'_Y (\phi_Y(t)) = f \tr \phi^{\tau,Y} (\tau_Y (t))  =  \phi^\tau ( f \tr \tau_Y(t))$  
for any $f\in C_0(Y)$. This equality and the  equation (\ref{Da:10}) imply that 
\[
\phi^\tau ( f \, \widehat{\e}_Y(t) ) \ = \
\phi^\tau ( f \tr \tau_Y(t)) \ = \
f \tr \tau'_Y( \phi_Y(t) )  \ = \
f \, \widehat{\e'}_Y( \phi_Y(t) ) \ ,
\]
completing the proof. 
%
%
%
\end{proof}

\

\noindent \textbf{Sections.}
A \emph{section} of $(\cA,\jmath)_K$ is defined as a collection  
$T:=\{T_Y\in\cA_Y\}_{Y\in K}$ 
satisfying the relation 
\[
T_{Y}= \jmath_{YV}(T_V) \ , \qquad V \leq Y \ . 
\]
\begin{proposition}
\label{prop.sec}
Let $T$ be a section of $(\cA,\jmath)_K$.
Then $\e^x_Y(T_Y) \in \cA_{x}$ is independent of $Y\in\omega_x$ for any $x \in X$, 
and $\wa T := \{ \e^x_Y(T_Y) \}_x$ defines a multiplier of $\sA$.
\end{proposition}

\begin{proof}
Let $Y,\tilde Y\in\omega_x$. Then there is $V \in \omega_x$ with $ V \subseteq Y \cap \tilde Y$ and
$$\e^x_Y(T_Y) = 
 \e^x_Y \circ \jmath_{YV}(T_V) = 
 \e^x_V(T_V) = 
 \e^x_{\tilde Y} \circ \jmath_{\tilde YV}(T_V) = 
 \e^x_{\tilde Y}(T_{\tilde Y}) \ ,$$
having applied (\ref{Da:00}).
So $\e^x_Y(T_Y) \in \cA_{x}$ is independent of $Y\in\omega_x$ and the vector field 
$\wa T := \{ \e^x_Y(T_Y) \} \in \cA_{x}$
is well defined. To verify that $\wa T$ defines a multiplier it suffices to make a check 
on the generators $f\,\widehat{\e}_Y(t)$, $f \in C_0(Y)$, $t \in \cA_Y$, $Y \in K$, of $\sA$:
\[
\wa T \cdot f\,\widehat{\e}_Y(t) \ = \ 
\{ \e^x_Y(T_Y) \cdot f(x) \, \widehat{\e}_Y(t)_x  \}_x \ = \
f \, \widehat{\e}_Y(T_Y t)
\ \in \sA \ .
\]
In the same way we find $f\,\widehat{\e}_Y(t) \cdot \wa T \in \sA$, 
so $\wa T$ is a multiplier as desired.
\end{proof}

\subsection{The case of injective nets} 
In the following lines we analyze the case of net bundles and, consequently, injective nets,
that are those of interest in algebraic quantum field theory, see \cite{RV12,BR09}. 
The following hypothesis that $X$ is paracompact is motivated by the use of transition maps, 
that work well for locally finite covers.
\begin{proposition}
\label{lem.b2}
Assume that  $X$ is also paracompact.
If $(\cA,\jmath)_K$ is a $\rC^*$-net bundle with fibre $A$
then $\sA$ is a locally constant continuous bundle with fibre $A$.
\end{proposition}

\begin{proof}
Without loss of generality we may assume that $\cA_Y = A$ for all $Y \in K$ and that 
$\jmath_{Y'Y} \in {\bf aut}A$ for any $Y \subseteq Y'$ (see \S \ref{sec.nets}).
%
%
Now, by definition the restriction $C_0(Y)\sA$, $Y \in K$, is generated by elements of the type
$f\,\widehat{\e}_V(t)$, $f \in C_0(V \cap Y)$, $t \in \cA_V = A$, $V \in K$, $V \cap Y \neq \emptyset$.
But since any $\jmath_{Y'Y}$ is an isomorphism, we may take $U \in K$, $U \subseteq V \cap Y$, and write
$t = \jmath_{VU} \circ \jmath_{UY}(t')$
for some $t' \in \cA_Y = A$, so
\[
C_0(Y) \sA \  \sim \  \{  f \widehat{\e}_Y(t) \ , \ f \in C_0(V \cap Y) , t \in A = \cA_Y , V \in K , V \cap Y \neq \emptyset  \} \ .
\]
Given a locally finite cover $\{ Y_\alpha \} \subset K$ we define the atlas
\[
\left\{
\begin{array}{ll}
\eta_\alpha : C_0(Y_\alpha) \sA \to C_0(Y_\alpha) \otimes A \ , \ 
\eta_\alpha(f\,\widehat{\e}_{Y_\alpha}(t)) := f \otimes t  \ , 
\\
\forall V \in K \ , \ f \in C_0(V \cap Y_\alpha)  \ , \ t \in A \ ,
\end{array}
\right.
\]
that is well-defined by injectivity of $\widehat{\e}_{Y_\alpha}$.
If $Y_\alpha \cap Y_\beta \neq \emptyset$ then there is $U \in K$, $U \subseteq Y_\alpha \cap Y_\beta$, 
so for any $f \in C_0(Y_\alpha \cap Y_\beta)$ and $t \in A = \cA_{Y_\beta}$ we find
\[
\eta_\alpha \circ \eta_\beta^{-1} (f \otimes t) \ = \ 
\eta_\alpha ( f\,\widehat{\e}_U \circ \jmath_{U Y_\beta}(t) ) \ = \ 
f \otimes \{ \jmath_{Y_\alpha U} \circ \jmath_{UY_\beta}(t) \} \ = \
f \otimes \theta_{\alpha \beta}(t) \ ,
\]
where 
\begin{equation}
\label{def.theta}
\theta_{\alpha \beta}:= \jmath_{Y_\alpha U} \circ \jmath_{UY_\beta}  \ , \qquad U\subset  Y_\alpha \cap Y_\beta \ . 
\end{equation}
Clearly $\theta_{\alpha \beta}\in {\bf aut}A$. Furthermore, since for any pair $U,V\subset  Y_\alpha \cap Y_\beta$ there is
$W\subset U\cap V$, using the net relations one can prove that the definition  $\theta_{\alpha\beta}$ is independent of the choice of $U$. 
In other terms $\{ \theta_{\alpha \beta} \}$ is a set of locally constant transition maps for $\sA$ 
and this concludes the proof \footnote{In \cite{RRV09} it is proved that $\{ \theta_{\alpha \beta} \}$
is a \v Cech cocycle and indeed a complete invariant of the net bundle.}.
\end{proof}

\begin{remark}
In the previous proof the transition maps of $\sA$ are defined by means
of (\ref{def.theta}) and this shows that $\sA$ is the locally constant bundle associated
with $(\cA,\jmath)_K$ in the sense of \cite[Prop.33]{RRV09}. 
\end{remark}

\begin{example}
\label{ex.lc}
Let $(\cBH,\ad U)_K$ be the $\rC^*$-net bundle of bounded
linear operators of the Hilbert net bundle $(\cH,U)_K$; then
the associated continuous $\rC^*$-bundle $\sBH$ is locally
constant and has fibre $B(\cH_a)$. If $\sH \to X$ is the
bundle of Hilbert spaces defined by $(\cH,U)_K$, 
then it is easily verified that $\sBH$ is the continuous 
$\rC^*$-bundle defined by $\sH$ in the sense of the following (\ref{eq.sbh}).
\end{example}

\begin{proposition}
\label{prop.fil01}
Let $(\cA,\jmath)_K$ be a net of $\rC^*$-algebras.
\begin{itemize}
\item[(i)] If  $(\cA,\jmath)_K$ is non-degenerate then $\sA$ is non-trivial.
\item[(ii)] If $(\cA,\jmath)_K$ is injective then $\tau:(\cA,\jmath)_K\nrightarrow\sA$ is a monomorphism.
\end{itemize}
\end{proposition}

\begin{proof}
$(i)$ By hypothesis there is a non-trivial representation
$\pi : (\cA,\jmath)_K \to (\cBH,\ad U)_K$,
so there is $Y \in K$ and $t \in \cA_Y$ such that 
$\pi_Y(t) \neq 0$.
Let $\bar \pi : \sA \to \sBH$ be the induced $C_0(X)$-morphism. 
Then for any $f \in C_0(Y)$ we have
$\bar \pi(f\,\widehat{\e}_Y(t)) = f \tr \pi_Y(t)$
(see (\ref{eq.uni}),
and since $\pi_Y(t) \neq 0$ we conclude that $\bar \pi(f\,\widehat{\e}_Y(t)) \neq 0$ as desired. 
$(ii)$ Recall that injectivity amounts to saying that the net admits a faithful representation $\pi$ (\S \ref{sec.nets}). So, 
assume that there is $t\in \cA_Y$ such that $t\ne 0$ and $\tau_Y(t) = 0$. Since $\pi_Y(t) \neq 0$ the equation (\ref{eq.uni}) 
implies that $f \tr \pi_Y(t) = \bar \pi(f\,\widehat{\e}_Y(t)) \neq 0$ for any  $f \in C_0(Y)$, $f \neq 0$. 
But $0 =\bar \pi(f \tr \tau_Y(t)) = \bar \pi(f \, \widehat{\e}_Y(t))$,  
and this leads to a contradiction.
So, $\tau_Y(t) \neq 0$ as desired.
\end{proof}

\subsection{$C_0(X)$-representations of nets}
\label{sec.rep}
In the following lines we analyze the consequences of the previous results
in the setting of representation theory on bundles of Hilbert spaces.\smallskip 

In this section we assume that $X$ is a connected, locally compact and  locally  simply connected 
Hausdorff space and that $K$ is a base of $X$, ordered under inclusion,  of open, relatively compact and simply connected subsets of 
$K$. Under these assumptions the fundamental group of $K$ is isomorphic to the fundamental group of $X$
(see Section \ref{sec.nets}).


\paragraph{On nets of Hilbert spaces.}
Let $H$ denote a separable Hilbert space.
As we saw in Rem.\ref{rem.b0}, we have a map
\begin{equation}
\label{eq.netBan}
(\cH,U)_K \mapsto (\sH,\eta)
\end{equation}
assigning to the Hilbert net bundle $(\cH,U)_K$ with fibre $H$ the locally constant Hilbert bundle $(\sH,\eta)$.
It is easily verified that (\ref{eq.netBan}) preserves tensor product and direct sums,
nevertheless it is not full, since the image of the set of morphisms between Hilbert net bundles is
the set of {\em locally constant} bundle morphisms.
This fact can be illustrated using Rem.\ref{rem.b0}:
the set of isomorphism classes of Hilbert net bundles over $K$ with fibre $H$ is given by
\[
\hom^\ad( \pi_1^a(K) , UH )
\ \simeq \
\hom^\ad( \pi_1(X) , UH )
\ ,
\]
where $\hom^\ad$ is the orbit space of the set of homomorphisms under the adjoint 
action of $UH$ and $x \in a \in K$. This set is, usually, drastically 
different from the set
$H^1(X,UH)$
of isomorphism classes of locally trivial Hilbert bundles over $X$;
for example, when $H$ is infinite-dimensional, $H^1(X,UH)$ is always
trivial by the Kuiper theorem, whilst
$\hom^\ad( \pi_1(X) , UH )$
is typically huge when the homotopy group $\pi_1(X)$ is not trivial.
Thus non-isomorphic Hilbert net bundles may define locally constant Hilbert bundles 
that become isomorphic in the larger category of Hilbert bundles (see \cite{RRV12} for details).

\begin{example}
\label{ex.rep}
$(i)$ 
     We take $X = S^1$ with $K$ the base of intervals $Y \subset S^1$ such that
     $\ovl Y \neq S^1$. Then $\pi_1^a(K) \simeq \pi_1(S^1) \simeq \bZ$ and
     \[
     \hom^\ad( \pi_1(S^1) , UH ) = 
     \hom^\ad( \bZ , UH ) = 
     UH^\ad 
     \ ,
     \]
     whilst it is well-known (\cite[\S 3.13]{Kar}) that
     $H^1(S^1,UH) = \{ 0 \}$.
$(ii)$
     Take $X=S^2$ and $K$ the base of disks $Y \subset S^2$ such that 
     $\ovl Y \neq S^2$. Then $\pi_1^a(K) = \{ 0 \}$, so that
     $\hom^\ad( \pi_1(S^2) , UH ) = \{ 0 \}$,
     whilst (by \cite[\S 3.13]{Kar} and, in the infinite dimensional case, the Kuiper theorem)
     \[
     H^1(S^2,UH)
     =
     \left\{
     \begin{array}{ll}
     \bZ \ \ ,\ \ \dim H < \infty
     \\
     \{ 0 \} \ \ , \ \ \dim H = \infty \ .
     \end{array}
     \right.
     \]
\end{example}


\paragraph{$C_0(X)$-representations.}
Let $\sH \to X$ be a bundle of Hilbert spaces. 
The set $\tilde \sH$ of continuous sections of $\sH$ vanishing at infinity has an obvious structure of
a Hilbert $C_0(X)$-bimodule (with coinciding left and right $C_0(X)$-actions). 
The $\rC^*$-algebra $\sKH$ of compact right $C_0(X)$-module operators of $\tilde \sH$ is a $C_0(X)$-algebra,
whilst the unital $\rC^*$-algebra $\sBHm$ of right $C_0(X)$-module adjointable operators on $\tilde \sH$ is endowed
with an embedding $C_0(X) \hookrightarrow \sBHm$, which, unless $X$ is compact, is degenerate, that is,
\begin{equation}
\label{eq.sbh}
\sBH \ := \ C_0(X) \sBHm \ \subsetneq \ \sBHm
\end{equation}
(for example, if $\sH = X \times H$ then $\sKH = C_0(X,K(H))$, $\sBH = C_0(X,B(H))$ and $\sBHm = C_b(X,B(H))$,
the $\rC^*$-algebra of bounded continuous maps from $X$ to $B(H)$).
%
Following Kasparov (\cite{Kas88}), given a $C_0(X)$-algebra $\sB$ we call $C_0(X)$-\emph{representation} 
a $C_0(X)$-morphism
\[
\pi : \sB \to \sBH \subseteq \sBHm
\]
(when $X$ is compact we take $\sBHm$ coinciding with $\sBH$).

\

Let $(\cA,\jmath)_K$ be a net with canonical morphism
$\tau : (\cA,\jmath)_K \nrightarrow  \sA$
and 
$\pi : \sA \to \sBHm$ 
a $C_0(X)$-representation. 
Applying (\ref{eq.comp}), we obtain the heteromorphism
\[
\pi_! := \pi \circ \tau : (\cA,\jmath)_K \nrightarrow  \sBH
\ .
\]
We call {\em $C_0(X)$-representations} of $(\cA,\jmath)_K$ heteromorphisms with values in $\sBH$.
The previous results show that any $C_0(X)$-representation of $\sA$ induces 
a $C_0(X)$-representation of $(\cA,\jmath)_K$.

Let us now consider a representation 
\begin{equation}
\label{eq.pi}
\pi : (\cA,\jmath)_K \to (\cBH,\ad U)_K \ .
\end{equation}
Then by Ex.\ref{ex.lc} there is a bundle of Hilbert spaces $\sH \to X$
such that $\sBH$ is the $C_0(X)$-algebra of $(\cBH,\ad U)_K$, so we have the 
canonical morphism 
$\e' : (\cBH,\ad U)_K  \nrightarrow  \sBH$
which yields the $C_0(X)$-representation
\[
\pi_! := \e' \circ \pi : (\cA,\jmath)_K \nrightarrow  \sBH \ ,
\]
inducing, by universality, the $C_0(X)$-representation
\begin{equation}
\label{eq.b1}
\pi^\tau : \sA \to \sBH 
\ \ : \ \
\pi^\tau \circ \e = \pi_!
\ .
\end{equation}
So we may regard representations of the type (\ref{eq.pi}) as particular cases of $C_0(X)$-representations.

\begin{remark}
Let $(\cH,U)_K$ denote a Hilbert net bundle defining the $\rC^*$-net bundle $(\cBH,\ad U)_K$.
The net structure $\ad U$ clearly restricts to isomorphisms at the level of 
compact operators
\[
\ad U_{Y'Y} : K(\cH_Y) \to K(\cH_{Y'})
\ \ , \ \
Y \subseteq Y'
\ ,
\]
thus we have the $\rC^*$-net subbundle $(\cK \cH,\ad U)_K$ of $(\cBH,\ad U)_K$ to which corresponds,
by functoriality, the inclusion of $C_0(X)$-algebras $\sKH \subseteq \sBH$. This implies that
each $K(\cH_Y)$, $Y \in K$, can be regarded as a set of local sections of $\sKH$.
\end{remark}

\section{$K$-homology and index}
\label{s.index}

In this section we define $K$-homology cycles for a net of $\rC^*$-algebras, 
that we call {\em nets of Fredholm modules}, and show that these yield cycles in representable $K$-homology,
i.e. continuous families of Fredholm modules. 
As mentioned in the introduction, our start point is the use of Fredholm modules made by Longo  
to describe sectors of a net (\cite[\S 6]{Lon01}), an idea that has been developed in terms
of nets of spectral triples for applications to conformal field theory (\cite{CHKL10}).

\

Let $K$ be a poset.
The Hilbert net bundle $(\cH,U)_K$ is said to be {\em graded} whenever there
is a self-adjoint, unitary section $\Gamma = \{ \Gamma_o \in U \cH_o \}$,
\begin{equation}
\label{eq.z2}
\Gamma_{o'} U_{o'o} = U_{o'o} \Gamma_o
\ \ , \qquad o \leq o' 
\ .
\end{equation}
This clearly yields a direct sum decomposition
$(\cH,U)_K = (\cH^+,U^+)_K \oplus (\cH^-,U^-)_K$.
Note that (\ref{eq.z2}) is equivalent to saying that $\Gamma$ is a section of 
$(\cBH , \ad U)_K$ inducing a $\bZ_2$-grading, thus we can consider the 
Banach net bundles of even/odd operators
\[
(\cB^\pm \cH , \ad U)_K 
\ \ : \ \ 
\ad \Gamma_o(t_\pm) = \pm t_\pm
\ , \
\forall t_\pm \in B^\pm(\cH_o) 
\ , \
\forall o \in K
\ .
\]
To be concise, we denote the family of inner automorphisms associated to $\Gamma$ 
by $\gamma = \{ \gamma_o := \ad \Gamma_o \}$. 
Let $(\cA,\jmath)_K$ be a net of $\rC^*$-algebras endowed with a period 2 automorphism
$\beta : (\cA,\jmath)_K \to (\cA,\jmath)_K$
(in this case we say that $(\cA,\jmath)_K$ is \emph{graded}), and $\pi$
a representation over the graded Hilbert net bundle $(\cH,U;\Gamma)_K$.
We say that $\pi$ is {\em graded} whenever
\begin{equation}
\label{eq.z2pi}
\gamma_o \circ \pi_o = \pi_o \circ \beta_o
\ \ , \ \
\forall o \in K
\ .
\end{equation}
The following definitions take into account the fact that each $\cA_o$, $o \in K$,
is unital; their formulation in the non-unital case may be given with the obvious
modifications.

\begin{definition}
\label{def.fred}
Let $(\cA,\jmath)_K$ be a graded/ungraded net of $\rC^*$-algebras. 
An even/odd \textbf{net of Fredholm modules} over $(\cA,\jmath)_K$ or, to be concise, 
a \textbf{Fredholm $(\cA,\jmath)_K$-module}, is a triple $(\pi,U;F)$, where:
\textbf{(i)}   $\pi$ is a nondegenerate, graded/ungraded 
            representation on the Hilbert net bundle $(\cH,U)_K$,
            having a grading $\Gamma$ in the even case;
\textbf{(ii)}  $F = \{ F_o \}$ is a section of $( \cBH , \ad U )_K$ 
            (i.e. $\ad U_{o'o}(F_o) = F_{o'}$ for all $o \leq o'$)
            and, in the even case,
            $\gamma_o(F_o) = - F_o$, $\forall o \in K$.
            It fulfils the relations
            \begin{equation}
            \label{eq.khom}
            F_o=F_o^*
            \ \ {\mathrm{and}} \ \
            F_o^2-1_o
            \ , \
            [F_o,\pi_o(t)] 
            \ \in K(\cH_o)
            \ \ , \ \ 
            \forall o \in K
            \ , \
            t \in \cA_o
            \ .
            \end{equation}
\end{definition}

We denote the set of Fredholm $(\cA,\jmath)_K$-modules by
$\sF(\cA,\jmath)_K$.

\smallskip

We now pass to the topological case and assume that $K$ is a good base for 
a $\sigma$-compact metrizable space $X$. 
For notation and terminology on representable $K$-homology we refer the reader 
to the appendix \S \ref{KK}. 
In the next results we interpret Fredholm $(\cA,\jmath)_K$-modules
in terms of continuous families of Fredholm operators.
\begin{lemma}
\label{lem.BH}
Let $(\cH,U)_K$ be a Hilbert net bundle and $T$ a section of $(\cBH,\ad U)_K$. 
Then the operator $\wa T$ defined as in Prop.\ref{prop.sec} is in $\sBHm$.
\end{lemma}

\begin{proof}
Let $\sH \to X$ denote the locally constant Hilbert bundle defined by $(\cH,U)_K$.
It is well-known that the multiplier algebra of $\sKH$ is $\sBHm$, 
so to prove the Lemma it suffices to verify that the vector field $\wa T$ defines a multiplier of $\sKH$.
To this end we consider generators $f \, \widehat{\e'}_Y(t)$ of $\sKH$, 
with $t \in K(\cH_Y)$, $f \in C_0(Y)$, $Y \in K$, and compute
\[
\wa T \cdot f \, \widehat{\e'}_Y(t) \ = \
\{ {\e'}^x_Y(T_Y) \cdot f(x) \, \widehat{\e'}_Y(t)_x  \}_x \ = \
f \, \widehat{\e'}_Y(T_Y t) \ ;
\]
since $T_Y t \in K(\cH_Y)$ we find $\wa T \cdot f \, \widehat{\e'}_Y(t) \in \sKH$,
and this proves $\wa T \in \sBHm$ as desired.
\end{proof}

\begin{theorem}
\label{thm.fred}
Let $(\cA,\jmath)_K$ be a net of $\rC^*$-algebras.
Then any Fredholm $(\cA,\jmath)_K$-module $(\pi,U;F)$ defines
a cycle in the representable $K$-homology $RK^0(\sA)$.
\end{theorem}

\begin{proof}
We focus on the even case (the odd one is analogous).
We start observing that, $(\cA,\jmath)_K \mapsto \sA$ being a functor, 
the $C_0(X)$-algebra $\sA$ is graded by $\beta^\tau \in \textbf{aut}\sA$.
By (\ref{eq.b1}) there is a locally constant Hilbert bundle $\sH$ defined by 
$(\cH,U)_K$ and a $C_0(X)$-representation
\[
\pi^\tau : \sA \to \sBH \ .
\]
By Lemma \ref{lem.BH}, $\Gamma$ and $F$ yield operators $\wa \Gamma , \wa F \in \sBHm$,
with $\wa \Gamma$ unitary and self-adjoint.
Now, $\sA$ is generated as a closed vector space by monomials of the type
\[
T := f_1 \widehat{\e}_{Y_1}(t_1) \cdots f_n \widehat{\e}_{Y_n}(t_n)
\ \ , \ \
Y_k \in K
\ , \
f_k \in C_0(Y_k)
\ , \
t_k \in \cA_{Y_k}
\ ;
\]
note that, by linearity, we may assume that $f_k \geq 0$ for all $k$.
We now verify the relations (\ref{eq.KK.khom}), starting from the commutativity up
to compact operators with elements of $\pi^\tau(\sA)$; to this end, by linearity
and density it suffices to verify on elements of the type $T$, and we proceed
inductively on the length $n$. 
As a first step we denote the generators of $\sBH$ by 
$f \, \widehat{\e'}_Y(T)$, $f \in C_0(Y)$, $T \in \cBH_Y$,
and compute (assuming $f_1 \geq 0$)
\[
\begin{array}{lcl}
[ \wa F  \ , \ \pi^\tau(f_1 \widehat{\e}_{Y_1}(t_1)) ]  & = &
[ f_1^{1/2} \wa F  \ , \ \pi^\tau(f_1^{1/2} \widehat{\e}_{Y_1}(t_1)) ]  = \\ & = &
[ f_1^{1/2} \widehat{\e'}_{Y_1}(F_{Y_1})  \ , \ \pi^\tau(f_1^{1/2} \widehat{\e}_{Y_1}(t_1)) ]  
\stackrel{ Cor.\ref{cor.b3} }{=} \\ & = &
%
%
%
[ f_1^{1/2} \widehat{\e'}_{Y_1}(F_{Y_1})  \ , \ f_1^{1/2} \widehat{\e'}_{Y_1}(\pi_{Y_1}(t_1)) ]  
= \\ & = &
f_1 \, \widehat{\e'}_{Y_1} \left( [ F_{Y_1} , \pi_{Y_1}(t_1) ]  \right)
\ \stackrel{(\ref{eq.khom})}{\in} \ \sKH \ .
\end{array}
\]
We now assume that $T$ is a monomial of length $n > 1$, define
$T_{\geq i} := f_i \widehat{\e}_{Y_i}(t_i) \cdots f_n \widehat{\e}_{Y_n}(t_n)$, $i \geq 2$,
and compute
\[
\begin{array}{l}
\left[ \wa F  , \pi^\tau(T) \right]  = \\ 
\left[ \wa F \ , \ \pi^\tau(f_1 \widehat{\e}_{Y_1}(t_1)) \cdot \pi^\tau(T_{\geq 2}) \right] = \\
\left[ \wa F \ , \ \pi^\tau(f_1 \widehat{\e}_{Y_1}(t_1)) \right] \, \pi^\tau(T_{\geq 2})  \ +
\pi^\tau(f_1 \widehat{\e}_{Y_1}(t_1)) \, \left[ \wa F \ , \  \pi^\tau(T_{\geq 2}) \right] \ .
\end{array}
\]
By induction the two summands of the previous expression belong to $\sKH$, and
this proves that $[ \wa F \ , \ \pi^\tau (T) ] \in \sKH$.
By linearity and density, we conclude that
$[ \wa F , \pi^\tau (T) ] \in \sKH$, $\forall T \in \sA$.
The other relations in (\ref{eq.KK.khom}) defining a Kasparov module are checked in the same way 
using (\ref{eq.khom}), and the theorem is proved.
\end{proof}

The previous result yields the desired geometric interpretation of objects, 
nets of Fredholm modules, that a priori are defined in terms of the abstract poset $K$.
From this point of view we consider Theorem \ref{thm.fred} satisfactory at the conceptual level, 
nevertheless we would like to point out that aspects related to the holonomy need a further investigation.
To illustrate this point we consider the case $X = S^1$.
By the Kuiper theorem we have that the Hilbert bundle $\sH$ of the previous theorem is trivial,
so we may regard $\wa F$ as a continuous map 
\[
\wa F : S^1 \to F(H) \ ,
\]
where $F(H) \subset B(H)$ is the space of Fredholm operators endowed with the norm topology.
By the Atiyah-Janich theorem $\wa F$ yields an element ${\mathrm{index}}(\wa F)$ of the $K$-theory $K^0(S^1)$, 
that is exactly the product (\ref{index}) of 
$[\sA]$, the $K$-class of the free, rank one right Hilbert $\sA$-module,  
by the Kasparov module $(\pi^\tau,\wa F)$.
But, as well-known, $K^0(S^1) = \bZ$, so ${\mathrm{index}}(\wa F) \in \bZ$
yields exactly the kind of invariant that we would obtain taking $\pi$ as a Hilbert space representation
and $\wa F$ as a constant map.
The reason of this fact relies in Ex.\ref{ex.rep}(i).
When we compute ${\mathrm{index}}(\wa F)$ we \emph{forget} the essential information of holonomy,
without which any locally constant bundle on $S^1$ appears simply as a trivial bundle.
So we need a tool that takes account of such a structure,
that is, the equivariant $K$-homology of the holonomy dynamical system defined
by $(\cA,\jmath)_K$ (see \cite[\S 3.3]{RV12}), that we study in \cite{RV2}.

\section{Applications and examples}
\label{sec.ex}

\paragraph{$\rC^*$-covers and nets.} 
Let $X$ be a compact Hausdorff space and $K$ a good base of $X$. 
A \emph{$\rC^*$-cover}, written $(Y,A)$, is given by a family of pairs 
$\{ (Y_i,A_i) \}_i$,
where:
\begin{itemize}
\item $Y = \{ Y_i \subset X \}$ is a (not open, in general) cover of $X$ such that $Y_i \cap Y_j = \emptyset$, $\forall i \neq j$;
\item $A$ is a $\rC^*$-algebra generated by the family of $\rC^*$-subalgebras $\{ A_i \subset A \}$.
\end{itemize}
The net of $\rC^*$-algebras $(\cA,\jmath)_K$ defined by $(Y,A)$ is given by
\[
\cA_o \ := \ \rC^* \{ A_i \, : \, Y_i \cap o \neq \emptyset \} \ \ , \ \  o \in K \ ,
\]
and the obvious inclusion morphisms $\jmath_{o'o} : \cA_o \to \cA_{o'}$, $o \subseteq o'$.

\begin{example}[Subtrivial $C(X)$-algebras]
A $C(X)$-algebra $\sB$ is said to be subtrivial whenever there is a 
$\rC^*$-algebra $B$ with a $C(X)$-monomorphism $i : \sB \to C(X) \otimes B$;
in this case we identify $\sB$ with its image through $i$, so that any fibre 
$\sB_x$ is a $\rC^*$-subalgebra of $B$. 
Let $\{ A_i \subseteq B \}$ denote the family of distinguished fibres of $\sB$.
Then we have the $\rC^*$-cover $(Y,A)$, where $A := \rC^* \{ A_i \} \subseteq B$ and
$Y_i := \{ x \in X  :  \sA_x = A_i \}$.
The $\rC^*$-cover $(Y,A)$ defines the net $(\cA,\jmath)_K$ and in accord with Theorem \ref{thm.b2},
the $C(X)$-algebra $\sA$; we have a $C(X)$-monomorphism
$\sB \to \sA$,
as, by construction, $\sB_x \subseteq \sA_x$, $\forall x \in X$.
The spirit is that any fibre $\sB_x$ is enlarged by those $\rC^*$-algebras $\sB_y$ where $x,y \in o$ for some $o \in K$,
thus $\sA$ has the propensity to "fill the gaps" of $\sB$. For example, we may take $X=[0,1]$ and
$\sB := \{ f \in C(X,A) : f(x) = 0 , \forall x \in [1/3,2/3] \}$
and easily verify that $\sA = C(X) \otimes A$.
\end{example}

\smallskip

In the following lines we discuss a class of examples that makes use, as an input, of
an inclusion $A_1 \subset A_2 = A$ of unital $\rC^*$-algebras and a *-automorphism $\alpha \in {\bf aut}A$.
We provide an interpretation of the above objects in terms of nets and $C_0(X)$-algebras,
where $X := S^1$ is endowed with the base $K$ of open intervals with proper closure in $X$. 
As a first step we consider the cover of $S^1$ (in polar coordinates)
\[
Y_1 \, = \, \{ \theta \in \left[ 0 , \pi \right]   \} 
\ \ , \ \
Y_2 \, = \, \{ \theta \in \left( \pi , 2\pi \right)   \} \ ;
\]
it yields the $\rC^*$-cover $(Y,A)$, that we may regard as the $\rC^*$-cover defined by the subtrivial $C_0(X)$-algebra
$\sB \, := \, \{ f \in C(S^1,A) \, : \, f(Y_1) \subseteq A_1 \}$.
The net $(\cA,\jmath)_K$ defined by $(Y,A)$ has fibres
\[
\begin{array}{ll}
\cA_o = A_i \ \ , \ \ {\mathrm{if}} \ \ o \subset Y_i \ {\mathrm{for \ a \ given}} \ i=1,2 \ , \\
\cA_o = A \ \ , \ \ {\mathrm{if}} \ \ o \cap Y_1 \neq \emptyset \ {\mathrm{and}} \  o \cap Y_2 \neq \emptyset \ ;
\end{array}
\]
moreover $\sA \simeq C(S^1,A)$, as can be verified by observing that
$\cA_x = A$, $\forall x \in S^1$.

Let us now consider a $K$-homology cycle $(\pi_\bullet,F_\bullet) \in RK^0(A)$, 
equivariant in the sense that $(\pi_\bullet,U_\bullet)$ is a covariant representation of $(A,\alpha)$ on the Hilbert space $H_\bullet$
and $[F_\bullet,U_\bullet] = 0$. Since $A$ is unital and ungraded, $(\pi_\bullet,U_\bullet)$ splits in accord to a direct sum decomposition
$H_\bullet = H_\bullet^0 \oplus H_\bullet^1$ 
such that
\[
F_\bullet \, = \,
\left(
\begin{array}{cc}
0 & T_\bullet^* \\
T_\bullet & 0
\end{array}
\right)
\ \ , \ \ 
T_\bullet : H_\bullet^0 \to H_\bullet^1
\ .
\]
We define the Hilbert net bundle $(\cH,U)_K$, where
\[
\cH_o = H_\bullet \ \ , \ \ 
U_{o'o} =
\left\{
\begin{array}{ll}
U_\bullet \ \ , \ \ {\mathrm{if}} \ \ \pi \in o' \ {\mathrm{and}} \ o \subset Y_1 \ , \\
1 \ \ , \ \  {\mathrm{otherwise}} \ ,
\end{array}
\right.  
\ \ , \ \ 
\forall o \subseteq o'
\]
(it is easy to verify that $U_{o''o} = U_{o''o'} U_{o'o}$, $\forall o \subseteq o' \subseteq o''$)
and the representations 
\[
\pi_o \,:= \,
\left\{
\begin{array}{ll}
\pi_\bullet \circ \alpha |_{\cA_o} \ \ , \ \ \forall o \in K \, : \, \pi \in o \ ,
\\
\pi_\bullet |_{\cA_o} \ \ , \ \  {\mathrm{otherwise}} \ .
\end{array}
\right. 
\]
It turns out that
\[
\pi_{o'} |_{\cA_o} = \ad U_{o'o} \circ \pi_o \ \ , \ \ \forall o \subseteq o' \ ,
\]
for example if $o \subset Y_1$ and $o \subset o'$, $\pi \in o'$, then 
$\pi_o = \pi_\bullet|_{A_1}$
and
$\pi_{o'} |_{\cA_o} =$ 
$\pi_\bullet \circ \alpha |_{A_1} =$ 
$\ad U \circ \pi_\bullet |_{A_1} =$ 
$\ad U_{o'o} \circ \pi_o$.
Thus
$\pi : (\cA,\jmath)_K \to (\cBH,\ad U)_K$
is a representation, that clearly has the grading inherited by $\pi_\bullet$.
Defining
$F_o := F_\bullet$, $T_o := T_\bullet$, $\forall o \in K$,
we obtain the Fredholm $(\cA,\jmath)_K$-module $(\pi,U;F)$.
Since $\ker T_\bullet$, $\ker T_\bullet^*$ are stable under the $U_\bullet$-action, 
we have finite-dimensional Hilbert net bundles $(E,v)_K$, $(E',v')_K$, where
$E_o := \ker T_o$, $v_{o'o} := U_{o'o}|_{E_o}$,
$E'_o := \ker T_o^*$, $v'_{o'o} := U_{o'o}|_{E'_o}$,
$\forall o \subseteq o'$.
These define locally constant Hilbert bundles 
$\sE \, , \, \sE' \to S^1$
respectively, and the index of the corresponding cycle $(\pi^\tau,\hat{F})$ in $RK^0(\sA)$ is given by 
\[
{\mathrm{index}}(\pi^\tau,\hat{F}) \, = \, [\sE] - [\sE'] \in K^0(S^1) \simeq \bZ \ .
\]
In general $(E,v)_K$, $(E',v')_K$ are not trivial net bundles
nevertheless, as noted at the end of the previous section, $\sE , \sE'$ are always trivial.

\paragraph{Quantum field theory and globally hyperbolic space-times.}

In the following paragraphs we assume that 
$X$ is a globally hyperbolic manifold with at least two spatial dimensions. 
We consider the base $K$ of \emph{diamonds} generating the topology of $X$, see \cite[Eq.3.2]{BR09} and,
in particular, when $X \simeq \bR^4$ is the Minkowski spacetime, the base $K$ of doublecones (\cite[\S I]{DHR71}). 
These are good in the sense of the present paper, thus the fundamental groups of $K$ and $X$ are isomorphic.

In the following paragraphs we show that some invariants
(in particular, the statistical dimension)
defined by nets of $\rC^*$-algebras
generated by quantum fields can be computed in terms of the index of a suitable Fredholm operator.
This idea has been suggested by Longo in \cite[\S 6]{Lon01} by using the concept of \emph{supercharge},
nevertheless we prefer to follow a different route, for the main reason that the known supercharges
correspond to "infinite-dimensional" Fredholm modules such as nets of spectral triples (\cite{CHKL10}), 
thus they cannot describe objects such as sectors with finite statistical dimension.

\paragraph{Shift operators and field nets.}
Let $(\cF,\jmath)_K$ be a net of von Neumann algebras defined on the separable Hilbert space $H$, that is,
any $\cF_o$ is a von Neumann subalgebra of $B(H)$ and $\jmath_{o'o}$, $o \subseteq o'$, is the inclusion.
We assume that the net acts irreducibly on $H$, i.e.\ the global algebra $\vec{\cF} := \vee_o \cF_o$ is irreducible\footnote{In the present section, the symbol $\vee$ stands for the von Neumann algebra
           generated by the given family of algebras.}, and that a strongly compact gauge group $G \subseteq U(H)$ acts on $(\cF,\jmath)_K$ by net automorphisms,
that is 
\[
\alpha_g(\cF_o) \ := \ g \cF_o g^* \ = \cF_o \ \ , \ \ g \in G \ , \ o \in K \ . 
\]
The construction of a net of the above type can be achieved by standard arguments.
One may take $X$ as a 4-dimensional globally hyperbolic spacetime and take, for example, the free Dirac field defined by \cite{Dim82}.
%
%
%
%
Returning to our general construction, the Hilbert space $H$ decomposes as follows  
\[
H \ = \ \bigoplus_{ \si \in {\mathrm{Irr}}G } H_\si \otimes L_\si \ ,
\]
where $H_\si$ is the multiplicity space and $L_\si$ is a finite-dimensional Hilbert space carrying
the irreducible representation $\si$, see \cite[\S 3.5]{GLRV}. 
We denote the spectral projections by 
$E_\si \in B(H)$, $\si \in {\mathrm{Irr}} \, G$; 
by construction, $E_\si H = H_\si \otimes L_\si$.
Spectral projections $E_\si$ define irreducible representations $\pi_\si$ 
(called \emph{topologically trivial superselection sectors}) of the fixed-point net
$(\cA,\jmath)_K$, $\cA_o := \cF_o^G$, $o \in K$: 
\[
\pi_\si : (\cA,\jmath)_K \to B(H_\si) 
\ : \ 
\pi_\si(t) \otimes 1_\si \ = \ E_\si t E_\si
\ \ , \ \ 
t \in \vec{\cA} := \vee_o \cA_o
\ , \ 
\si \in {\mathrm{Irr}} \, G \ ,
\]
where $1_\si \in B(L_\si)$ is the identity.
Let now us consider a vector $\omega \in H_\si$, $\| \omega \| = 1$, with the spirit that $\omega$ is a reference state of $(\cA,\jmath)_K$ inducing $\pi_\si$ as the associated GNS representation.
Then we may pick an ordering on the base of $H_\si$ such that $\omega$ is the first element of the base,
and define the corresponding shift operator $S_\omega \in B(H_\si)$, $S_\omega^* S_\omega = 1$, 
such that $P_\omega := S_\omega S_\omega^*$ is the projection onto $H_\si \ominus \bC \omega$.

\smallskip

We now perform a twist of the representations $\pi_\si$, obtaining superselection sectors affected by the topology of $X$ as in \cite{BR09}.
For simplicity we assume that $\pi_1(X)$ is abelian, a condition fulfilled in the interesting cases of 
de Sitter spacetimes ($X \simeq S^n \times \bR$, $n \in \bN$) and anti-de Sitter spacetimes ($X \simeq S^1 \times \bR^n$, $n \in \bN$).
We denote by $\pi_1(X)^*$ the character space of $\pi_1(X)$ and consider the set of pairs
\[
\varrho := (\si,\chi) \in {\mathrm{Irr}} \, G \times \pi_1(X)^* \ .
\]
By \cite[\S 5.1]{RRV12} any $\chi \in \pi_1(X)^*$ defines a rank 1 Hilbert net bundle $(\cL^\chi,Z^\chi)_K$,
thus any $\varrho \in {\mathrm{Irr}} \, G \times \pi_1(X)^*$ defines the Hilbert net bundle $(\cH^\varrho,U^\varrho)_K$, where
\[
\cH^\varrho_o := H_\si \otimes \cL^\chi_o \otimes L^\chi_\si
\ \ , \ \ 
U^\varrho_{o'o} := id(H_\si) \otimes Z^\chi_{o'o} \otimes 1_\si
\ \ , \ \
o \subseteq o'
\ ,
\]
and the representation
\[
\pi^\varrho : (\cA,\jmath)_K \to (\cBH^\varrho,\ad U^\varrho)_K
\ \ , \ \ 
\pi_o(t) := \pi_\si(t) \otimes id(\cL^\chi_o) \otimes 1_\si
\ , \
t \in \cA_o
\ , \ 
o \in K
\ .
\]
We denote the associated graded representation by $(\hat{\pi}^\varrho,\hat{U}^\varrho)$, 
where 
$\hat{\cH}^\varrho := \cH^\varrho \oplus \cH^\varrho$,
$\hat{U}^\varrho := U^\varrho \oplus U^\varrho$,
$\hat{\pi}^\varrho := \pi^\varrho \oplus \pi^\varrho$.
Let now us consider the $\rC^*$-algebras
$\cA^{\si,\omega}_o := \{ t \in \cA_o \ : \ [ \pi_\si(t) , S_\omega ] \in K(H_\si) \}$, $o \in K$,
that contain at least the multiples of $1 \in \cap_o \cA_o$.
Then we have the net $(\cA^{\si,\omega},\jmath)_K$ with associated $C_0(X)$-algebra $\sA^{\si,\omega}$, 
and the Fredholm $(\cA^{\si,\omega},\jmath)_K$-module $(\hat{\pi}^\varrho,\hat{U}^\varrho;F^\omega)$ where
\[
F^\omega_o \, := \,
\left(
\begin{array}{cc}
0 & S_\omega^* \otimes id(\cL_o) \otimes 1_\si \\
S_\omega \otimes id(\cL_o) \otimes 1_\si & 0
\end{array}
\right)
\ \ , \ \ 
o \in K
\ .
\]
Note that ${\mathrm{index}}F^\omega_o = \ker (S_\omega^* \otimes id(\cL_o) \otimes 1_\si)$ can be identified with $L_\si$ 
and has the structure of a $G$-module, thus ${\mathrm{index}}F^\omega_o$ can be regarded as a $G$-index.
It is clear that the $G$-module structure is preserved by the adjoint action $\ad \hat{U}^\varrho_{o'o}$, $o \subseteq o'$,
and the same holds when passing to the continuous family of Fredholm operators associated with $F^\omega$ in the sense of Theorem \ref{thm.fred}.

\smallskip

We conclude that for any superselection sector labelled by 
$\varrho = (\si,\chi) \in {\mathrm{Irr}} \, G \times \pi_1(X)^*$ 
there is a net of Fredholm modules $(\hat{\pi}^\varrho,\hat{U}^\varrho;F^\omega)$ carrying the additional structure given by the $G$-action.
We assign to $(\hat{\pi}^\varrho,\hat{U}^\varrho;F^\omega)$ the corresponding cycle 
$( \hat{\pi}^{\varrho,\tau},\hat{F}^\omega) \in RK^0_G(\sA^{\si,\omega})$ 
and consider the index
\[
{\mathrm{index}} \, (\hat{F}^\omega) \in RK^0_G(X) \ .
\]
Let $R(G)$ denote the representation ring of $G$ with elements classes $[\si]$, $\si \in {\mathrm{Rep}} \, G$;
applying to the above index the morphisms 
$* : RK^0_G(X) \to R(G)$, $f : R(G) \to \bZ$ 
given by the evaluation over a fixed $x \in X$ and the map forgetting the $G$-module structure respectively,
we obtain the indices
\[
{\mathrm{index}}_*(\hat{F}^\omega) = [\si] \in R(G) 
\ \ , \ \ 
{\mathrm{index}}_f(\hat{F}^\omega) = {\mathrm{dim}}L_\si \in \bZ \ .
\]
The above indices are independent of $\omega \in H_\si$;
in particular, ${\mathrm{index}}_f(\hat{F}^\omega)$ is by definition the statistical dimension of $\pi^\varrho$.

\paragraph{The net bundle of a superselection sector.}
Let $(\cA,\jmath)_K$ be a net of von Neumann algebras in the Hilbert space $H_0$,
where any $\jmath_{o'o} : \cA_o \to \cA_{o'}$, $o \subseteq o'$, is the inclusion.
Differently from the previous paragraph, now the idea is that elements of $\vec{\cA} := \vee_o \cA_o$ are operators interpreted as quantum \emph{observables}
acting on the reference separable Hilbert space $H_0$.
In this paragraph we assume that $X$ is the Minkowski spacetime, thus $H_0$ is the \emph{vacuum Hilbert space}.
%

\smallskip

Now, a superselection sector of $(\cA,\jmath)_K$ is labelled by a unitary family (\emph{cocycle})
\[
u_{o'o} \in U(\cA_{o'}) \ \ , \ \ o \subseteq o' \in K
\ \ {\mathrm{such \ that}} \ \ 
u_{o''o} = u_{o''o'} u_{o'o} \ \ , \ \  o \subseteq o' \subseteq o'' \ ,
\]
that we denote by $u$, see \cite[\S 4]{BR09} or \cite[\S 3.4.7]{Rob90}
{\footnote{
In the above-cited reference $u$ is defined in a different way as a unitary family indexed by elements of the first simplicial set $\Si_1(K)$.
Here, to be concise, we give a definition in terms of ordered pairs $o \subseteq o' \in K$;
the fact that the two definitions agree is proved in \cite[Prop.4.4]{RRV09}.}.
Any cocycle $u$ defines a family of *-endomorphisms 
$\varrho = \{ \varrho_o \in {\bf end}\vec{A} \}$, $\vec{A} := \vee_o \cA_o$,
(called the \emph{charge}), such that 
\begin{equation}
\label{eq.ex0}
\varrho_{o'} \, = \, \ad u_{o'o} \circ \varrho_o
\ \ , \qquad  o \subseteq o'
\ ,
\end{equation}
see \cite[\S 3.4.7]{Rob90}. We consider the intertwiner spaces
\begin{equation}
\label{eq.ex1}
(\varrho_o^r,\varrho_o^s) 
\, := \, 
\{ t \in \vec{A} \, : \, t \varrho_o^r(T) = \varrho_o^s(T)t \, , \, \forall T \in \vec{A} \}
\ \ , \ \ 
r,s \in \bN \ , \ o \in K \ ,
\end{equation}
where by definition $\varrho_o^0 := id_{\vec{A}}$.
The spaces (\ref{eq.ex1}) are sets of arrows of a tensor *-category $\cT^\varrho_o$ with objects $\varrho_o^r$, $r \in \bN$, 
see \cite[Ex.1.1]{DR89}. Moreover, (\ref{eq.ex0}) induces tensor *-isomorphisms
\[
\hat{u}_{o'o} : \cT^\varrho_o \to \cT^\varrho_{o'} 
\ : \
\left\{
\begin{array}{ll}
\hat{u}_{o'o}(\varrho_o^r) \, := \, \varrho_{o'}^r \ , \ r \in \bN \ , \\
\hat{u}_{o'o}(t) \, := \, u^{(s)}_{o'o} \, t \, (u^{(r)}_{o'o})^* \ , \ t \in (\varrho_o^r,\varrho_o^s) \ ,
\end{array}
\right.
\]
where $u^{(r)}_{o'o} := \varrho_{o'}^{r-1}(u_{o'o}) \cdots \varrho_{o'}(u_{o'o}) u_{o'o}$, $\forall r \in \bN$.
Consider the family of \emph{Doplicher-Roberts $\rC^*$-algebras}: 
$\cO^\varrho_o :=$ $\rC^* \{ (\varrho^r,\varrho^s) \}_{r,s}$, $o \in K$ (see \cite[\S 4]{DR89}).
Then $\cO^\varrho_o \subseteq \cA_o$, $\forall o \in K$, see \cite[\S IV]{DHR71}.
Moreover, applying \cite[Theorem 5.1]{DR89}, the isomorphisms $\hat{u}_{o'o}$ induce *-isomorphisms
\[
\hat{u}_{o'o} : \cO^\varrho_o \to \cO^\varrho_{o'} \ \ , \ \ o \subseteq o' \ ,
\]
fulfilling the net relations (note the abuse of the notation $\hat{u}_{o'o}$).
Thus we have the $\rC^*$-net bundle $(\cO^\varrho,\hat{u})_K$, that is trivial because $\pi_1(X)$ is trivial.

We pick a fibre $a \in K$.
When $\varrho$ is special in the sense of \cite[p.66]{DR90} and $u$ has statistical dimension 
$d$, there is an isomorphism $\cO^\varrho_a \simeq \cO_{\si(G)}$, \cite[Theorem 4.17]{DR89}, where 
$G$ is the compact group of global gauge symmetries of $(\cA,\jmath)_K$, 
$\si : G \to {\mathbb{SU}}(d)$ is a unitary representation inducing a *-automorphic action of $G$ on the Cuntz algebra $\cO_d$
with fixed-point algebra $\cO_{\si(G)}$.
Thus, $\cO^\varrho_a$ is a unital, simple (\cite[Theorem 3.1]{DR87}), separable and nuclear (\cite{DLRZ02}) $\rC^*$-algebra.

By triviality, representations of $(\cO^\varrho,\hat{u})_K$ correspond to Hilbert space representations of $\cO^\varrho_a$
and odd Fredholm $(\cO^\varrho,\hat{u})_K$-modules correspond to odd $K$-homology cycles for $\cO^\varrho_a$.
The $K$-homology
$RK^1(\cO^\varrho_a) \simeq  RK^1(\cO_{\si(G)})$
can be computed by using the methods in \cite{PR96,KP97}, anyway in the following lines we show that 
odd Fredholm $(\cO^\varrho,\hat{u})_K$-modules can be constructed using only the representation 
\[
\pi : \cO^\varrho_a \to B(H_0) \ \ , \ \ \pi := \varrho_a |_{ \cO^\varrho_a } \ ,
\]
that is, the restriction to $\cO^\varrho_a$ of the superselection sector defined by $u$. 
To this end, the crucial remark is that the nets used in quantum field theory verify the Borchers property \cite[p.61]{DR90}: 
given $a\in K$ and  a projection $E\in\cA_a$, for any $a'\in K$ containing the closure of $a$  
there is an isometry $V\in\cA_{a'}$ such that $V^*V=1$ and $VV^* = E$, thus $E$ is necessarily infinite. 
On these grounds, we note that:
\begin{itemize}
\item since $\cO^\varrho_a$ is a unital subalgebra of $\cA_a$ and $\varrho_a$ is unital, $\pi$ is non-degenerate;
\item by the Borchers property, the $\rC^*$-algebra $\cO^\varrho_a$ has no nontrivial compact operators. 
\end{itemize}
The above properties say that $\pi$ \emph{ample} in the sense of \cite[Chap.5]{HR}.
Furthermore, since $\cO^\varrho_a$ is unital, separable and nuclear, by \cite[Prop.5.1.6, Theorem 8.4.3]{HR} we obtain the isomorphism 
$RK^1(\cO^\varrho_a) \simeq K_0(\cD \cO^\varrho_a)$, 
where 
\[
\cD\cO^\varrho_a 
\, := \, 
\{ T \in B(H_0) \, : \, [ \varrho_a(t) , T ] \in K(H_0) \, , \, \forall t \in \cO^\varrho_a  \} \ .
\]
This shows that $RK^1(\cO^\varrho_a)$ is determined by $\pi$ as claimed.

\smallskip

\indent The case of an arbitrary globally hyperbolic spacetime $X$ deserves a further study that we postpone to a future work. 
In general, the charges $\varrho_o$, $o \in K$, cannot be defined over the global algebra $\vec{A}$ (see \cite[\S 5.1]{BR09}),
and the net $(\cO^\varrho,\hat{u})_K$ is not trivial, with the consequence that the r\^ole of $\cD\cO^\varrho_a$ is played by a presheaf.

\appendix

\section{Basics of representable $KK$-theory}
\label{KK}

For reader convenience in this section we recall some notions of representable $KK$-theory 
(see \cite[\S 2]{Kas88} for details). 
Let $\sA$ be a $C_0(X)$-algebra.
A grading on $\sA$ is given by a $C_0(X)$-automorphism $\gamma$ of $\sA$ with period $2$;
a $C_0(X)$-morphism between graded $C_0(X)$-algebras is said to be graded whenever 
it intertwines the relative automorphisms.
Given the graded $C_0(X)$-algebra $(\sB,\gamma')$, a Hilbert $\sB$-module $H$ is said
to be graded whenever, for all $v,w \in H$, $b \in \sB$, 
there is a linear map $\Gamma : H \to H$ such that
\[
\Gamma (vb) = (\Gamma v) \gamma'(b)
\ \ , \ \
(\Gamma v , \Gamma w) = \gamma'(v,w) \in \sB \ .
\]
Let now $\sA,\sB$ be graded $C_0(X)$-algebras, 
eventually endowed with the trivial grading. 
A {\em Kasparov} $\sA$-$\sB$-bimodule, denoted by $(\eta,\phi)$, is given by:
(i) a graded Hilbert $\sB$-bimodule $H$ carrying a graded representation
$\eta : \sA \to B(H)$
(here $B(H)$ is the $\rC^*$-algebra of adjointable, right $\sB$-module operators), 
such that
\[
\eta(fa)vb = \eta(a)v(fb)
\ \ , \ \
\forall v \in H
\ , \ 
a \in \sA
\ , \ 
b \in \sB
\ , \
f \in C_0(X)
\ .
\]
(ii) an operator $\phi \in B(H)$ such that
\begin{equation}
\label{eq.KK.khom}
(\phi - \phi^*)\eta(t) 
\ , \
(\phi^2 - 1)\eta(t)
\ , \
[\phi,\eta(t)] 
\ \in K(H)
\ \ , \ \
\forall t \in \sA
\ ,
\end{equation}
where $K(H) \subseteq B(H)$ is the ideal of compact $\sB$-module operators.
We denote the set of {\em Kasparov} $\sA$-$\sB$-bimodules by $E(\sA,\sB)$.
When $\sB = C_0(X)$, $H$ is, in essence, a (separable) continuous field of Hilbert spaces.
We say that $(\eta_0,\phi_0) , (\eta_1,\phi_1) \in E(\sA,\sB)$ are homotopic
whenever there is $(\eta,\phi) \in E(\sA,\sB \otimes C([0,1]))$ such that
$(\eta_0,\phi_0) , (\eta_1,\phi_1)$ are obtained by applying to $(\eta,\phi)$ 
the evaluation morphism 
$\sB \otimes C([0,1]) \to \sB$
over $0,1 \in [0,1]$ respectively.
The representable $KK$-theory $RKK(X;\sA,\sB)$ is defined as the abelian group
of homotopy classes of Kasparov $\sA$-$\sB$-bimodules w.r.t. the
operation of direct sum. 
In particular, we define
\[
RK^0(\sA) := RKK(X;\sA,C_0(X)) 
\ \ , \ \
RK_0(\sA) := RKK(X;C_0(X),\sA) 
\ ;
\]
these groups are called the {\em representable $K$-homology of $\sA$}
and, respectively, the {\em representable $K$-theory of $\sA$}. 
By \cite[Prop.2.21]{Kas88}, the Kasparov product induces the pairing
\begin{equation}
\label{index}
\left \langle \cdot , \cdot \right \rangle :
RK_0(\sA) \times RK^0(\sA) \to RK^0(X) := RK_0(C(X)) \ ,
\end{equation}
in essence the map defined by the index of continuous families of Fredholm operators.
When $X$ is compact $RK^0(X)$ is the topological $K$-theory (see \cite[Prop.2.20]{Kas88}).


\end{document}